\documentclass[10pt]{article}
\let\conjugatet\overline

\usepackage{color}
\usepackage{amsmath}

\usepackage{footnote}
\usepackage{subscript}
\usepackage{amsfonts,color}
\usepackage{amsmath,amssymb}
\usepackage{amssymb} 
\usepackage{mathtools}
\usepackage{amsthm,wasysym}
\usepackage{framed}
\usepackage{cancel}
\usepackage{graphicx}
\usepackage{caption}
\usepackage{subcaption}
\usepackage{rotating}

\usepackage[english]{babel}
\usepackage[utf8]{inputenc}
\makeatletter
\usepackage{float}
\usepackage{wrapfig}
\restylefloat{figure}

\makeatletter
\let\@fnsymbol\@arabic
\makeatother

\usepackage{bbm}

\newcommand{\id}{{\mathbf{\mathbbm{1}}}}

\newcommand{\tr}{{\rm tr}}
\newcommand{\dev}{{\rm dev}}
\newcommand{\Sym}{{\rm Sym}}
\newcommand{\sym}{{\rm sym}}
\newcommand{\skw}{{\rm skew}}

\newcommand{\Div}{{\rm Div}}
\newcommand{\axl}{{\rm axl}}
\newcommand{\anti}{{\rm \textbf{Anti}}}
\newcommand{\so}{\mathfrak{so}}

\newcommand{\norm}[1]{\|#1\|}

\def\dd{\displaystyle}

\newtheorem{theorem}{Theorem}[section]
\newtheorem{lemma}[theorem]{Lemma}

\newtheorem{proposition}[theorem]{Proposition}

\newtheorem{definition}[theorem]{Definition}
\def\barr{\begin{array}}
\usepackage{lscape}

	\def\earr{\end{array}}
\def\bec#1{\begin{equation}\label{#1}}
\def\becn{\begin{equation*}}
\def\endec{\end{equation}}
\def\endecn{\end{equation*}}
\def\dd{\displaystyle}
\def\bfm#1{\mbox{\boldmat}}

\begin{document}
	
	\title{\vspace*{-1.75cm}   {Propagation} of Love waves in  {linear elastic} isotropic  Cosserat materials}
	
\author{Marius Apetrii\thanks{Marius Apetrii,\ Department of Mathematics,  Alexandru Ioan Cuza University of Ia\c si,  Blvd.
		Carol I, no. 11, 700506 Ia\c si,
		Romania; and  Octav Mayer Institute of Mathematics of the
		Romanian Academy, Ia\c si Branch,  700505 Ia\c si, email:  mapetrii@uaic.ro},\quad  Emilian Bulgariu\thanks{ Emilian Bulgariu,  {``Ion Ionescu de la Brad'' Iasi University of Life Sciences, Romania,  email:  emilian.bulgariu@iuls.ro}}, \quad  Ionel-Dumitrel Ghiba\thanks{ Ionel-Dumitrel Ghiba, \ Department of Mathematics,  Alexandru Ioan Cuza University of Ia\c si,  Blvd.
		Carol I, no. 11, 700506 Ia\c si,
		Romania; and  Octav Mayer Institute of Mathematics of the
		Romanian Academy, Ia\c si Branch,  700505 Ia\c si, email:  dumitrel.ghiba@uaic.ro}, \qquad Hassam Khan\thanks{Hassam Khan,    Department of Mathematics and Statistics,  Institute of Business and Management Sciences (IOBM), Karachi. Pakistan, email: dr.hassam@iobm.edu.pk}\vspace{2mm}\\  
		 {and} \quad  Patrizio Neff\,\thanks{Patrizio Neff,  \ \ Head of Lehrstuhl f\"{u}r Nichtlineare Analysis und Modellierung, Fakult\"{a}t f\"{u}r
		Mathematik, Universit\"{a}t Duisburg-Essen,  Thea-Leymann Str. 9, 45127 Essen, Germany, email: patrizio.neff@uni-due.de} }
\date{\it to appear in Proceedings A
	The Royal Society}
\maketitle
\begin{abstract}
	
We investigate the propagation of Love waves in an isotropic half-space modelled as a linear  {elastic isotropic} Cosserat  material. To this aim, we show that a method commonly used to study Rayleigh wave propagation is also applicable to the analysis of Love wave propagation. This approach is based on the explicit solution of an algebraic Riccati equation, which operates independently of the traditional Stroh formalism. The method provides a straightforward numerical algorithm to determine the wave amplitudes and speed{s}. Beyond its numerical simplicity, the method guarantees the existence and uniqueness of a subsonic wave speed, addressing a problem that remains unresolved in most Cosserat solids generalised  {continua} theories. Although often overlooked, proving the existence of an admissible solution is, in fact, the key point that validates or invalidates the entire analytical approach used to derive the equation determining the wave speed. Interestingly, it is confirmed that the Love waves do not need the artificial introduction of 
a surface layer, as indicated in the literature.

  \medskip
  
  
  \medskip
  
  \noindent\textbf{AMS 2020 MSC:} 74J15, 74M25, 74H05, 74J05
  
\end{abstract}

\begin{footnotesize}
	\tableofcontents
\end{footnotesize}

\section{Introduction}
Wave propagation in complex materials has been a subject of strong interest due to its significance in fields such as seismology, material science, and engineering.  {It} is important to study the propagation of surface waves in other media, not only on Earth surface, due to their applications in tomography and in the construction and simulation of various metamaterials which are able to absorb the effect of such waves.

Among the various types of surface waves, the two types of surface waves are named Love waves \cite{Love1911} and Rayleigh waves \cite{Rayleigh}. There is a fundamental difference between these two type of surface waves. Love waves have only a horizontal motion that moves the surface from side to side perpendicular to the direction the wave is traveling, but not up and down, while the Rayleigh waves appear to roll like waves on an ocean,  it moves the ground up and down, and forward and backward in the direction that the wave is moving.  For both these two type of surface waves, it is considered that the boundary of the  half-space modelling the body is free of surface traction. Beside these, it is considered that the non-vanishing displacements and stresses decay to zero with respect to the  {depth}.

Love waves are of particular importance because of their ability to propagate along the surface of elastic media with distinct polarisation and dispersion properties. In seismology, Love waves are seismic surface waves responsible for the horizontal shifting of the Earth's surface during an earthquake and they  are observed in experiments and measurements, but theoretically confirmed in the framework of clasical model of  isotropic and homegeneous elastic materials only when when a low-velocity layer overlies a higher-velocity layer or sub-layers. The motion of particles in a Love wave follows a horizontal line perpendicular to the direction of propagation, making them transverse waves.  The amplitude, or maximum particle motion, tends to decrease rapidly with depth but not  along the Earth's surface. The slower decay rate along the Earth's surface makes surface waves, including Love waves, persist longer over distances compared to body waves, which dissipate faster as they propagate in three dimensions. Large earthquakes can produce Love waves that travel around the Earth multiple times before fading away.
Due to their slow decay, Love waves are the most destructive seismic waves outside the immediate vicinity of an earthquake's focus or epicentre. These waves are often what people feel most strongly during an earthquake.  {However, our} study is not only about the Love waves propagation along the Earth, it is rather a theoretical study which may be applied to an arbitrary medium modelled with the help of  {the} linear Cosserat theory \cite{Cosserat09,Eringen64,Mindlin64}. 

 The study of Love waves  in generalized continuum theories, such as Cosserat materials \cite{Cosserat09,Eringen64,Mindlin64,Eringen99}, presents an intriguing opportunity to extend our understanding of wave behaviour beyond the classical framework of  {linear} elasticity. Cosserat materials, also known as micropolar media, incorporate additional degrees of freedom, including rotational motions and couple stresses, allowing them to capture size effects and micro-structural influences that are absent in traditional continuum models. These distinctive features make Cosserat materials particularly suited for modelling media with internal structures, such as composites, granular materials, and metamaterials. In the framework of the Cosserat theory, the  propagation of Rayleigh waves in an isotropic Cosserat elastic half space  was studied in \cite{ChiritaGhiba3} using  the Stroh formalism \cite{destrade2007seismic,stroh1962steady}.  The strong point of the approach  in \cite{ChiritaGhiba3}  is that  explicit
 expressions of the attenuating coefficients and explicit conditions upon wave
 speed were found, as well as the exact expressions of three linear independent
 amplitude vectors. Then,  a simple form of the secular
 equation is written, which,  by comparison  with other
 generalized forms of the secular equation for Cosserat materials {\cite{Eringen99,Kulesh01,Kulesh03,Kulesh05,Erofeyev,Koebke}}  does not involve the complex form of
 the attenuating coefficients. However, from \cite{ChiritaGhiba3}  it is not obvious how to avoid in  {general} the spurious roots of the secular equation, as long as the Stroh formalism is used.  For Rayleigh waves it is shown in   \cite{khan2022existence} that the new form of the secular equations, written in terms of the impedance matrix, does not admit any spurious root, i.e. there exists only one subsonic surface wave solution of the secular equation.   {Going back to the Love waves, their propagation in isotropic elastic Cosserat materials has been previously  studied in \cite{kundu2017love,kulesh2009problem,Kulesh1,Kulesh2}. 
 Indeed,  \cite{kulesh2009problem,Kulesh1,Kulesh2} seem to be the  first studies in which  a qualitatively new wave mode  is observed  for the Cosserat elastic model (actually they are waves of  Love-type).  
 However,  one important part of the study is yet missing, i.e., the proof that the secular equation really has  {an} admissible solution that ensures that the  {strength of} the Love wave decreases with depth. In our present study this fact is rigorously proved and confirms the remark made in \cite{kulesh2009problem,Kulesh1,Kulesh2}.}

 To this aim we use the method introduced by Fu and Mielke \cite{fu2002new} and Mielke and Fu \cite{mielke2004uniqueness} in classical linear elasticity.
Motivated by  previous work of Mielke and Sprenger \cite{mielke1998quasiconvexity} which is more related to  control theory \cite{knobloch2012topics} than to surface wave propagation, Fu and Mielke \cite{fu2002new} and Mielke and Fu \cite{mielke2004uniqueness} have devised a new method for anisotropic elastic materials which is not based on the Stroh formalism, it is conceptually different from the other methods   and it is mathematically  well explained. They have shown that the {\it impedance matrix} \cite{ingebrigtsen1969elastic} defining the secular equation is the solution of an \textit{algebraic Riccati equation}. Using the properties of this equation, in \cite{mielke2004uniqueness,fu2002new} it is then  {shown} that   the secular equation does not admit spurious roots. It can be  said that the works by Mielke and Fu give an elegant final  answer to the general case of anisotropic materials, in the framework of classical linear elasticity. For the classical anisotropic model, the same form of this equation is present in the paper by Mielke  and Sprenger \cite{mielke1998quasiconvexity}  for $v=0$ and in the papers by Fu and Mielke \cite{mielke2004uniqueness,fu2002new} in the case of general $v$, see also the earlier work by Biryukov \cite{biryukov1985impedance}. 
 In \cite{khan2022existence} we have used  Fu and Mielke's method \cite{fu2002new,mielke2004uniqueness} in the study of the propagation of Rayleigh wave in linear elastic Cosserat  materials.  One of the reviewers' question at that time was if it is not possible to do the  {same} complete study but for the Love wave propagation. Our expectation was that the entire approach presented  by Fu and Mielke \cite{mielke2004uniqueness,fu2002new} should be suitable in the Cosserat theory for Love wave propagation, too. We mention that while the study of Love wave propagation in classical linear isotropic elastic materials is easier than the study of  {Rayleigh} waves since it reduce to a simple ordinary differential equation, at the first look, for Cosserat materials the difficulty of the studies of Love waves propagation and of the Rayleigh wave propagation have at least the same difficulty level. In fact, during our  {investigation} we  {observed} that in some points of our analysis, the Love wave propagation problem comes with some additional difficulties,  {here, determining} the analytical form of the limiting speed. However,  {\it we prove both the existence and the uniqueness} of the wave speed of the Love waves in Cosserat materials. 

 The structure  of the present paper is now  {as follows}. In Section \ref{Sp} we  {present} the system  {of} partial differential  {equations} which describe the behaviour of the material in the linear theory of isotropic and homogeneous Cosserat elastic materials and we give  {necessary} and sufficient conditions for real plane waves propagation. In Section \ref{setLove} we present the setup of the propagation of Love waves. In Section \ref{anzsec} we define and prove the existence of the limiting speed, we obtain the Riccati equation and the secular equation for the  {wave} speed propagation, and we give our main result which  {proves} the existence and uniqueness of the solution of the secular equation in the range fixed by the limiting speed. We mention that this final step is  {rarely} possible in generalised models if other methods are used. 
In Section \ref{Num1} we provide a numerical algorithms which can be implemented for any  material once   the constitutive coefficients are known. We present effective  numerical results for  dense polyurethane  {Foam@0.18mm} with the constitutive parameters given in \cite{Lakes83,lakes1987foam}. We remark that while the propagation speed of the Rayleigh waves  is slower than the propagation speed of the Love waves along the Earth's surface, in the dense polyurethane Foam@0.18mm it is vice versa.  In the last section we give a comparative study with the classical elasticity theory of homogeneous isotropic materials.

\section{Preliminaries}\label{Sp}\setcounter{equation}{0}

We consider that the mechanical behaviour of a body occupying  the unbounded regular region of three dimensional Euclidean space  is modelled with the help of the Cosserat theory of  linear isotropic elastic materials. We denote  by $n$ the outward unit normal on $\partial\Omega$. The body is referred to a fixed system of rectangular Cartesian axes  $Ox_i \,(i=1,2,3)$ , $\{e_1, e_2, e_3\}$ being the unit vectors of these axes. For further notations please see the Appendix \ref{NS}

\subsection{The  {linear isotropic} Cosserat model for elastic solids}\label{subsC}

In the Cosserat model for isotropic materials two vector fields are used to describe the macro- and micro-behaviour of the solid body, i.e.,  
the  {classical} displacement $u:\Omega \subset\mathbb{R}^3\mapsto\mathbb{R}^3$ and the   microrotation vector field $ \vartheta   :\Omega \subset\mathbb{R}^3\mapsto\mathbb{R}^3$, $ \vartheta   ={\rm axl}\, \mathbf{A}$, where { $\mathbf{A}={\rm Anti}\, (\vartheta), \mathbf{A}\in \mathfrak{so}(3)$} represents the microrotation tensor in the  Cosserat theory. In the framework of the linear isotropic hyperelastic theory the Cosserat model is  based on the elastic energy density   \cite{NeffGhibaMicroModel,MadeoNeffGhibaW,MadeoNeffGhibaWZAMM,madeo2016reflection,NeffGhibaMadeoLazar}
 \begin{align}\label{energy}
 	W(\mathrm{D}u ,{\rm Anti}\, \vartheta , \mathrm{D} \vartheta    )=& \underbrace{\mu_{\rm e} \,\lVert \dev_3\, \sym\,\mathbf{e}\rVert ^{2}+\mu_{\rm c}\,\lVert \skw\,\mathbf{e}\rVert ^{2}+\frac{2\,\mu_{\rm e} +3\,\lambda_{\rm e} }{6}\left[\mathrm{tr} \left(\mathbf{e}\right)\right]^{2}}_{:=W_1(\mathbf{e})}\notag\\& +\underbrace{{\mu_{\rm e}L_{\rm c}^2}\,\left[\ {\alpha_1}\,\lVert \dev_3\,\sym\,\mathbf{\mathfrak{K}}\rVert ^{2}+{\alpha_2}\,\lVert \skw\,\mathbf{\mathfrak{K}}\rVert ^{2}+\frac{2\,\alpha_1+3\,\alpha_3}{6}\left[\mathrm{tr} \left(\mathbf{\mathfrak{K}}\right)\right]^{2}\right]}_{:=W_2(\mathbf{\mathfrak{K}})},
 \end{align}
 where  we have used the following definitions for the independent constitutive  variables $\mathbf{e}$  (the non-symmetric strain tensor) and $\mathbf{\mathfrak{K}}$ (the curvature tensor) \begin{align}
 	\mathbf{e}:=\mathrm{D}u -\mathbf{A}=\mathrm{D}u -{\rm Anti}\, (\vartheta)   , \qquad \qquad \mathbf{\mathfrak{K}}:=\mathrm{D} \vartheta.
 \end{align}
 
 Moreover, the stress-strain relations for the homogeneous  isotropic Cosserat elastic solid are
 \begin{align}\label{01}
 	\boldsymbol{\sigma}&:=\frac{\partial\, W}{\partial\, \mathbf{e}}=2\,\mu_{\rm e} \, \sym \,\mathbf{e}+2\,\mu_{\rm c} \, \skw\, \mathbf{e}+\lambda_{\rm e} \,\tr (\mathbf{e})\, \id,\notag\\ 
 	\mathbf{ m}&:=\frac{\partial\, W}{\partial\, \mathbf{\mathfrak{K}}} ={\mu_{\rm e}L_{\rm c}^2}\,\left[2\,{\alpha_1}\, \sym\,\mathbf{\mathfrak{K}} +2\,{\alpha_2}\,\skw\, \mathbf{\mathfrak{K}}+\alpha_3\,\tr(\mathbf{\mathfrak{K}})\,\id\right],
 \end{align}
 where $\boldsymbol{\sigma}$ is the non-symmetric force stress tensor and $\mathbf{m}$ is the second-order non-symmetric couple stress tensor. 
 
The constants $(\mu_{\rm e},\lambda_{\rm e}) $,  $\mu_{\rm c}, L_{\rm c} $  and $( \alpha_1, \alpha_2, \alpha_3)$  are the elastic moduli  representing the parameters related to the meso-scale,  the
 parameters related to the micro-scale the Cosserat couple modulus, the characteristic length, and the three
 general isotropic curvature parameters (weights), respectively.   
  Due to the orthogonal Cartan-decomposition of the Lie-algebra $\mathfrak{gl}(3)\cong\mathbb{R}^{3\times 3}$, the strict positive definiteness of the potential energy is equivalent to the following simple relations for the introduced parameters  \cite{NeffGhibaMicroModel}
 \begin{align}
 	\quad\mu_{\rm e}  >0, \qquad \quad\mu_{\rm c}>0,\qquad \quad  2\,\mu_{\rm e} +3\,\lambda_{\rm e} >0,\qquad \quad \alpha_1>0, \qquad \quad\alpha_2>0, \qquad \quad 2\,\alpha_1+3\,\alpha_3>0 . \label{posCoss}
 \end{align}
 However, our entire subsequent analysis will be made under weaker conditions on the constitutive parameters,  { see \eqref{d12}}.

In the absence of  external  body forces and of   external  body moment, the Euler-Lagrange equations constructed  {based} on this internal energy are given as the following   the PDE-system of the model  \cite{Eringen99}
\begin{align}\label{eqisaxl}
	\rho\,u_{,tt}&=\Div[2\,\mu_{\rm e} \, \sym\,\mathrm{D}u +2\,\mu_{\rm c} \,\skw(\mathrm{D}u -\mathbf{A})+ \lambda_{\rm e}\, \tr(\mathrm{D}u ){\cdot} \id]\, ,\notag\\
	2\,\rho\, {j\,\mu_{\rm e}\,}\,\tau_{\rm c}^2\,\,\,(\axl\, \,\mathbf{A})_{,tt}&={\mu_{\rm e}L_{\rm c}^2}\,\bigg[{\alpha_1 }\,\dev\,\sym (\mathrm{D}\, \axl  \,\mathbf{A})+{\alpha_2}\,\skw (\mathrm{D}\, \axl  \,\mathbf{A})+ \alpha_3\,\tr(\mathrm{D}\, \axl  \,\mathbf{A}){\cdot } \id\bigg]
	\\&\ \ \ \ \ \ -4\,\mu_{\rm c} \,\axl\,(\skw\,\mathrm{D}u -\mathbf{A})\,  \ \ \ \text {in}\ \ \  \Omega\times [0,T],\notag
\end{align}
 which rewritten in indices read
\begin{align}\label{PDE}
	\rho\,\frac{\partial^2 \,u_i}{\partial\, t^2}&=\underbrace{(\mu_{\rm e} +\mu_{\rm c})\dd\sum_{l=1}^3\frac{\partial^2 u_i }{\partial x_l^2}+(\mu_{\rm e} -\mu_{\rm c}+\lambda_{\rm e} )\dd\sum_{l=1}^3\frac{\partial^2 u_l }{\partial x_l\partial x_i}+2\,\mu_{\rm c} \,\sum_{l,s=1}^3\varepsilon _{ils}\frac{\partial \vartheta _s }{\partial x_l}}_{={\rm Div}\,\boldsymbol{\sigma}},
	\vspace{1.2mm}\\
	\rho\, j\,\mu_{\rm e}\,\tau_{\rm c}^2\,\frac{\partial^2 \,\vartheta_i}{\partial\, t^2}&=\underbrace{{\mu_{\rm e}L_{\rm c}^2}\,\left[(\alpha_1+\alpha_2)\dd\sum_{l=1}^3\frac{\partial^2 \vartheta _i }{\partial x_l^2}+(\alpha_1-\alpha_2+\alpha_3)\dd\sum_{l=1}^3\frac{\partial^2 \vartheta_l }{\partial x_l\partial x_i}\right]}_{=\,{\rm Div}\,\mathbf{ m}}+2\,\mu_{\rm c} \,\sum_{l,s=1}^3\varepsilon _{isl}\frac{\partial u_l }{\partial x_s}-4\, \mu_{\rm c}\, \vartheta _i, \notag
\end{align}
with $i=1,2,3$, where  {$\rho$ is the density,  $j$ is an inertia weight parameter,  $\tau_{\rm c}$ the internal characteristic time (then the combination $j\,\mu_{\rm e}\,\tau_{\rm c}^2$ describes the mass density of inertia moment) \cite[page 163]{Eringen99}}  and $\varepsilon _{ils}$ is the Levi-Civita symbol.
	These equations are in complete agreement to the equations proposed in the Cosserat theory \cite{Eringen99}, but written in indices, see \cite{ghiba2023cosserat} for a complete comparison.

In the following we assume $\rho>0$ and $j>0$ without mentioning these conditions in the hypothesis of our results.
\subsection{Real plane waves in the direction $\xi = (\xi_1,\xi_2,0)^T$ }\label{Rpw}

{Even if  we will  consider the  propagation of  surface waves with the direction $e_1 =
	(1,0,0)^T$, i.e. some horizontal direction, $x_2$ being the vertical direction orthogonal
	to the surface along which the wave decays, the method we use in this paper needs  to consider a general direction of wave propagation $\xi = (\xi_1,\xi_2,0)^T$ and to characterise  where the wave is
	a real bulk wave. Therefore, for  $\xi=(\xi_1,\xi_2,0)^T$ with $\lVert{\xi}\rVert^2=1$}, it is useful to {know} for which conditions on the constitutive parameters we have that for every  wave number $k>0$  the system of partial differential equations \eqref{PDE} admits a non trivial solution in the form
\begin{align}\label{rwr}
	u(x_1,x_2,t)&=\begin{footnotesize}\begin{pmatrix}0\\0\\\widehat{u}_3\end{pmatrix}\end{footnotesize}
	\, e^{{\rm i}\, \left(k\langle {\xi},\, x\rangle_{\mathbb{R}^3}-\,\omega \,t\right)}\,,\qquad\qquad  \vartheta (x_1,x_2,t)=\begin{footnotesize}{\rm i}\,\begin{pmatrix}\widehat{\vartheta }_1\\\widehat{\vartheta }_2\\0\end{pmatrix}\end{footnotesize} \, e^{{\rm i}\, \left(k\langle {\xi},\, x\rangle_{\mathbb{R}^3}-\,\omega \,t\right)},\\&
	(\widehat{u}_3,\widehat{\vartheta }_1, \widehat{ \vartheta_2   })^T\in\mathbb{C}^{3}, \quad (\widehat{u}_3,\widehat{\vartheta }_1, \widehat{ \vartheta_2   })^T\neq 0\,\notag
\end{align}
only for real positive values $\omega^2$, where ${\rm i}\,=\sqrt{-1}$ is the complex unit.  

However, since our formulation is isotropic, by  demanding real plane waves in any direction $\xi=(\xi_1,\xi_2,0)$, $\lVert \xi\rVert=1$, it is equivalent  {to demanding} real plane waves in the direction $e_1=(1,0,0)$. Inserting  \eqref{rwr} into \eqref{PDE} we see that   $\widehat{u}_3, \widehat{\vartheta}_1$ and $\widehat{\vartheta}_2$ have to satisfy the system of linear  {algebraic} equations
\begin{align}\label{algPDE2}
-\omega^2\rho\, \widehat{u}_3&=-k^2\,(\mu_{\rm e} +\mu_{\rm c})\dd \widehat{u}_3 -2\,k\,\mu_{\rm c}\,\widehat{\vartheta }_2 ,\notag
\vspace{1.2mm}\\
-{\rm i}\,\omega^2\rho\, j\,\mu_{\rm e}\,\tau_{\rm c}^2\,\widehat{ \vartheta }_1&=-{\rm i}\,k^2\,{\mu_{\rm e}\,L_{\rm c}^2}\,(\alpha_1+\alpha_2)\,\widehat{ \vartheta }_1-{\rm i}\,k^2\,{\mu_{\rm e}\,L_{\rm c}^2}\,(\alpha_1-\alpha_2+\alpha_3)\widehat{\vartheta }_1 -4\,{\rm i}\, \mu_{\rm c}\, \widehat{\vartheta }_1,\vspace{1.2mm}\\
-{\rm i}\,\omega^2\rho\, j\,\mu_{\rm e}\,\tau_{\rm c}^2\,\widehat{ \vartheta }_2&=-{\rm i}\,k^2\,{\mu_{\rm e}\,L_{\rm c}^2}\,(\alpha_1+\alpha_2)\,\widehat{ \vartheta }_2-2\,{\rm i}\,k\,\mu_{\rm c}\,\widehat{u}_3 -4\,{\rm i}\, \mu_{\rm c}\, \widehat{\vartheta }_2. \notag
\end{align}

In \cite{khan2022existence} we find that  for all $k>0$  the system \eqref{PDE}$_2$
admits non trivial solutions $w=\begin{footnotesize}\begin{pmatrix}
\widehat{u}_3,\widehat{ \vartheta }_1,\widehat{ \vartheta }_2
\end{pmatrix}\end{footnotesize}^T$  only for  real positive values $\omega^2$  if and only if the following eigenvalue problem admits only real solutions
\begin{align}
&\left[\underbrace{\begin{footnotesize}\begin{pmatrix}
	k^2\frac{ \mu_{\rm e} +\mu_{\rm c} }{\rho}& 0& -2\, k\, \frac{\mu_{\rm c}}{\rho\, \sqrt{j\,\mu_{\rm e}\,\tau_{\rm c}^2\,}}\vspace{2mm}\\
	0& \ k^2\, {\mu_{\rm e}\,L_{\rm c}^2}\,\frac{2\,\alpha_1 +\alpha_3}{\rho\, j\,\mu_{\rm e}\,\tau_{\rm c}^2\,}+4\, \frac{\mu_{\rm c}}{\rho \, j\,\mu_{\rm e}\,\tau_{\rm c}^2}& 0\vspace{2mm}\\
	-2\, k\, \frac{\mu_{\rm c}}{\rho\, \sqrt{j\,\mu_{\rm e}\,\tau_{\rm c}^2\,}}& 0&\ \ k^2\, {\mu_{\rm e}\,L_{\rm c}^2}\,\frac{\alpha_1 +\alpha_2}{\rho\, j\,\mu_{\rm e}\,\tau_{\rm c}^2\,}+4\, \frac{\mu_{\rm c}}{\rho \, j\,\mu_{\rm e}\,\tau_{\rm c}^2\,}
	\end{pmatrix}\end{footnotesize}}_{:=\,\widetilde{\mathbf{Q}}_2(e_1,k)}-\omega^2\id\right]\, h=0, 
\end{align}
with $h=\widehat{\id}^{1/2}\begin{footnotesize}\begin{pmatrix}
\widehat{u}_3\\\widehat{ \vartheta }_1\\\widehat{ \vartheta }_2
\end{pmatrix}\end{footnotesize}$ and	$\widehat{\id} =\begin{footnotesize}\begin{pmatrix} \rho&0	&0
\\
0&\rho\,j\,\mu_{\rm e}\,\tau_{\rm c}^2
&0
\\0
& 0 & \rho\,j\,\mu_{\rm e}\,\tau_{\rm c}^2\end{pmatrix}\end{footnotesize}$, i.e., if and only if 
\begin{align}\label{rw3}
\mu_{\rm e}>0, \qquad\qquad  \mu_{\rm c}>0,  \qquad\qquad 2\, \alpha _1+\alpha _3 >0, \qquad\qquad (\alpha _1+\alpha _2) >0.
\end{align}

Then, for $k>0$ and due to the assumed isotropy, extrapolating to all directions of propagation, we have
\begin{proposition}\label{proprealw2} {\rm \cite{khan2022existence}} Let   ${\xi}\in \mathbb{R}^3$, $\xi\neq 0$ be any direction of the form $\xi=(\xi_1,\xi_2, 0)^T$. The necessary and sufficient conditions for existence of   a non trivial solution of the system of partial differential equations \eqref{PDE}  of the form given by \eqref{rwr} are \footnote{  {These conditions have physical interpretations, i.e., \(\mu_{\rm e} > 0\) is the stability condition (for the existence of transverse waves at low frequencies that propagate with velocity \(\sqrt{\frac{\mu_{\rm e}}{\rho}}\), \(2\alpha_1 + \alpha_3 > 0\) is the condition for the existence of longitudinal rotational wave at high frequencies,
		 \(\alpha_1 + \alpha_2 > 0\) is the condition for the existence of the acoustic branch of the transversal waves at high frequencies, while $\mu_{\rm c}>0$ means that the Cosserat and classical elastic effects are not uncoupled.}}
	\begin{align}\label{d12}
	\mu_{\rm e} >0,\qquad \qquad \mu_{\rm c} >0,\qquad \qquad   \alpha_1+\alpha_2>0 \qquad \qquad 2\alpha_1+\alpha_3 >0.
	\end{align}
\end{proposition}
The  restrictions \eqref{d12}  will be imposed for the rest of the paper.  Note that these conditions do not involve the constitutive parameter $\lambda_{\rm e}$ and do not imply the existence of real waves in any directions.

\section{The setup for the propagation of Love waves}\label{setLove}\setcounter{equation}{0}

 In the framework of the Love wave, we consider the region $\Omega$ to be the half space $$\Sigma:=\{(x_1,x_2,x_3)\,|\,x_1,x_3\in \mathbb{R},\,\, x_2\geq0\}.$$ The boundary of the homogeneous and isotropic half-space  is free of surface traction, i.e.,
 \begin{align}\label{03}
 \boldsymbol{\sigma}.\, n=0,\qquad \qquad \mathbf{m}.\, n=0 \qquad \qquad\textrm{for}\quad x_2=0.
 \end{align}
 
 \begin{figure}[h!]
 	\centering 
 		\includegraphics[angle=0,width=0.75\linewidth]{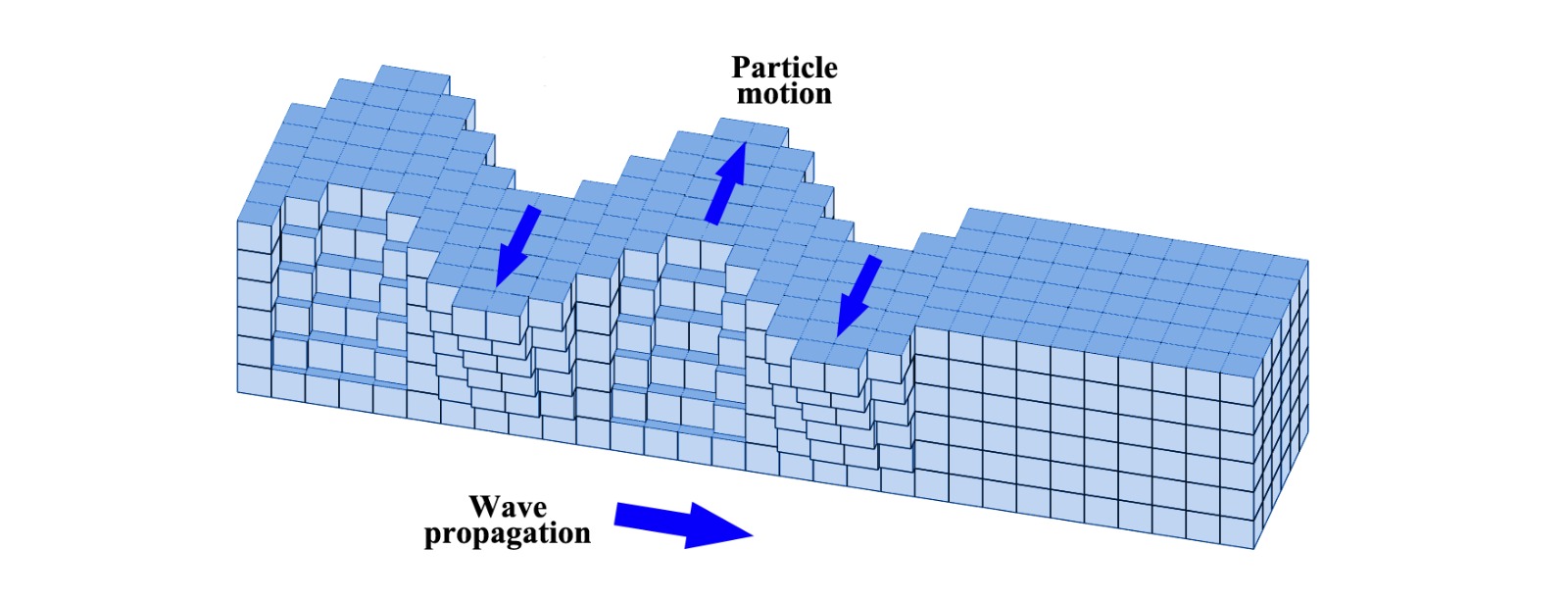}\vspace*{-0.5cm}
 		\caption{Love waves propagation.}\label{fig-love}
 \end{figure}

 In addition, the solution has to satisfy the following   decay condition
 \begin{align}\label{04}
 \lim_{x_2 \rightarrow \infty}\{u_3,{\vartheta_1}, \vartheta_2,\sigma_{13},\sigma_{23},m_{33}\} (x_1,x_2,t)=0 \qquad \quad \forall\, x_1\in\mathbb{R}, \quad \forall \, t\in[0,\infty).
 \end{align}

In isotropic solids and in the context of Love wave propagation, the surface particles move in the planes normal to the surface $x_2=0$ and  {orthogonal} to the direction of propagation  $e_1=(1,0,0)^T$. Accordingly to these characteristics of the Love waves, we consider the following   {plane} strain ansatz as a first step in our process of construction of the solution
 \begin{align}
 u(x_1,x_2,t)&=\begin{footnotesize}\begin{pmatrix} 0
 	\\
 	0
 	\\u_3(x_1,x_2,t)
 	\end{pmatrix}\end{footnotesize},\qquad \qquad \qquad\qquad \qquad{\rm D}u(x_1,x_2,t)=
 	\begin{footnotesize}\begin{pmatrix}
 	0 & 0 & \frac{\partial\, u_{3}}{\partial\,x_1}(x_1,x_2,t)\\
 	0 & 0 & \frac{\partial\, u_{3}}{\partial\,x_2}(x_1,x_2,t)\\
 	0 & 0 & 0
 	\end{pmatrix}\end{footnotesize}, \\ 
 	\mathbf{A}(x_1,x_2,t)&=
 	\begin{footnotesize}\begin{pmatrix}
 	0&0&\vartheta_2(x_1,x_2,t)\\
 	0&0&-\vartheta_1(x_1,x_2,t)\\
 	-\vartheta_2(x_1,x_2,t)&\vartheta_1(x_1,x_2,t)&0
 	\end{pmatrix}\end{footnotesize},\qquad \quad
 	\vartheta (x_1,x_2,x_3,t)=\axl \,\mathbf{A}=\begin{footnotesize}\begin{pmatrix} \vartheta   _1(x_1,x_2,t)
 	\\
 	\vartheta   _2(x_1,x_2,t)
 	\\ 0
 		\end{pmatrix}\end{footnotesize}.\notag
 \end{align}
 
Corresponding to our ansatz, we deduce 
the following form of the stress tensor 
 \begin{align}\label{stress_tens}
 \boldsymbol{\sigma}=
\begin{footnotesize}\begin{footnotesize}\begin{pmatrix}
0 & 0 & (\mu_{\rm e}+\mu_{\rm c})\,\frac{\partial\, u_{3}}{\partial\,x_1}-2\,\mu_{\rm c}\,\vartheta _2
\vspace{2mm} \\
0 & 0 & (\mu_{\rm e}+\mu_{\rm c})\,\frac{\partial\, u_{3}}{\partial\,x_2}+2\,\mu_{\rm c}\,\vartheta _1
\vspace{2mm} \\
 (\mu_{\rm e}-\mu_{\rm c})\,\frac{\partial\, u_{3}}{\partial\,x_1}+2\,\mu_{\rm c}\,\vartheta _2
 &
  (\mu_{\rm e}-\mu_{\rm c})\,\frac{\partial\, u_{3}}{\partial\,x_2}-2\,\mu_{\rm c}\,\vartheta _1
 &	
0
 \end{pmatrix}\end{footnotesize}\end{footnotesize},
 \end{align}
 and of the couple stress tensor 
 \begin{align}\label{couple_stress_tens}
\mathbf{m}= {\mu_{\rm e}\,L_{\rm c}^2}\,
\begin{footnotesize}\begin{footnotesize}\begin{pmatrix}
({2\alpha_1+\alpha_3})\,\frac{\partial \,\vartheta _{1}}{\partial \,x_1}+\alpha_3\,\frac{\partial \,\vartheta _{2}}{\partial \,x_2}
 &
({\alpha_1-\alpha_2})\,\frac{\partial \,\vartheta _{1}}{\partial \,x_2}+({\alpha_1+\alpha_2})\,\frac{\partial \,\vartheta _{2}}{\partial \,x_1}
 &
0
\vspace{2mm}\\
({\alpha_1+\alpha_2})\,\frac{\partial \,\vartheta _{1}}{\partial \,x_2}+({\alpha_1-\alpha_2})\,\frac{\partial \,\vartheta _{2}}{\partial \,x_1}
 &
 \alpha_3\,\frac{\partial \,\vartheta _{1}}{\partial \,x_1}+({2\alpha_1+\alpha_3})\,\frac{\partial \,\vartheta _{2}}{\partial \,x_2}
 &
0
 \vspace{2mm}\\
 0
 &
0
 &	
 \alpha_3\,\frac{\partial \,\vartheta _{1}}{\partial \,x_1}+\alpha_3\,\frac{\partial \,\vartheta _{2}}{\partial \,x_2}
 \end{pmatrix}\end{footnotesize}\end{footnotesize},
 \end{align}
 while the equation of motion  {is} reduced to 
 \begin{align}\label{49}
 \rho\,\frac{\partial^2 \, {u}_{3}}{\partial\,t^2}&=\frac{\partial \,\sigma_{13}}{\partial \,x_1}+\frac{\partial \,\sigma_{23}}{\partial \,x_2},\notag\\
 \rho\,j\,\mu_{\rm e}\,\tau_{\rm c}^2\, \,\frac{\partial^2 \,{\vartheta _{1}}}{\partial\,t^2}&= \frac{\partial \,m_{11}}{\partial \,x_1}+ \frac{\partial \,m_{21}}{\partial \,x_2}-\sigma_{23}+\sigma_{32},\\
\rho\,j\,\mu_{\rm e}\,\tau_{\rm c}^2\, \,\frac{\partial^2 \,{\vartheta _{2}}}{\partial\,t^2}&= \frac{\partial \,m_{12}}{\partial \,x_1}+ \frac{\partial \,m_{22}}{\partial \,x_2}+\sigma_{13}-\sigma_{31},\notag
 \end{align}
 subjected to the aforementioned boundary conditions  \eqref{03}, which turn out to be
 \begin{align}\label{50}
 \sigma_{23}=0,\qquad m_{21}=m_{22}=0 \qquad \text{at}\qquad x_2=0.
 \end{align}

 Therefore, our further aim  is to give an explicit solution  $(u_3, \vartheta_1,\vartheta_2   )$ of the following  system
 \begin{align}\label{x6}
 \rho\, \frac{\partial^2{u}_{3}}{ \partial \,t^2}&=(\,\mu_{\rm e}+\mu_{\rm c}) \,\left(\frac{\partial^2 \,u_{3}}{\partial \,x_1^2}+\frac{\partial^2 \,u_{3}}{\partial \,x_2^2}\right)-2\,\mu_{\rm c} \, \frac{\partial \,\vartheta _{2}}{\partial \,x_1}+2\,\mu_{\rm c} \, \frac{\partial \,\vartheta _{1}}{\partial \,x_2},\\
\rho \,j\,\mu_{\rm e}\,\tau_{\rm c}^2\,\frac{\partial^2 \,{{\vartheta }_{1}}}{\partial \,t^2}&={\mu_{\rm e}\,L_{\rm c}^2}\left[\,(2\alpha_1+\alpha_3) \,\frac{\partial^2 \,\vartheta _{1}}{\partial \,x_1^2}+(\alpha_1+\alpha_2) \,\frac{\partial^2 \,\vartheta _{1}}{\partial \,x_2^2}+(\alpha_1-\alpha_2+\alpha_3) \,\frac{\partial^2 \,\vartheta _{2}}{\partial \,x_1\,\partial \,x_2}\right]-2\,\mu_{\rm c} \,  \frac{\partial \,u_{3}}{\partial \,x_2}-4\,\mu_{\rm c} \, \vartheta _1,\notag\\
\rho \,j\,\mu_{\rm e}\,\tau_{\rm c}^2\,\frac{\partial^2 \,{{\vartheta }_{2}}}{\partial \,t^2}&={\mu_{\rm e}\,L_{\rm c}^2}\left[\,(\alpha_1+\alpha_2) \,\frac{\partial^2 \,\vartheta _{2}}{\partial \,x_1^2}+(2\alpha_1+\alpha_3) \,\frac{\partial^2 \,\vartheta _{2}}{\partial \,x_2^2}+(\alpha_1-\alpha_2+\alpha_3) \,\frac{\partial^2 \,\vartheta _{1}}{\partial \,x_1\,\partial \,x_2}\right]+2\,\mu_{\rm c} \,  \frac{\partial \,u_{3}}{\partial \,x_1}-4\,\mu_{\rm c} \, \vartheta _2,\notag
 \end{align}
 which satisfies the boundary conditions at $x_2=0$
 \begin{align}\label{x17}
  (\mu_{\rm e}+\mu_{\rm c})\,\frac{\partial\, u_{3}}{\partial\,x_2}+2\,\mu_{\rm c}\,\vartheta _1=&\,0,\notag\\
 {\mu_{\rm e}\,L_{\rm c}^2}\,({\alpha_1+\alpha_2})\,\frac{\partial \,\vartheta _{1}}{\partial \,x_2}+{\mu_{\rm e}\,L_{\rm c}^2}\,({\alpha_1-\alpha_2})\,\frac{\partial \,\vartheta _{2}}{\partial \,x_1}=&\,0,\\
 {\mu_{\rm e}\,L_{\rm c}^2}\,\alpha_3\,\frac{\partial \,\vartheta _{1}}{\partial \,x_1}+{\mu_{\rm e}\,L_{\rm c}^2}\,({2\alpha_1+\alpha_3})\,\frac{\partial \,\vartheta _{2}}{\partial \,x_2}=&\,0,\notag
 \end{align}
 and which has the asymptotic behaviour \eqref{04}.

 \section{The ansatz for the solution and the limiting speed}\label{anzsec}\setcounter{equation}{0}
 We look for a solution of \eqref{x6} and \eqref{x17} having the form\footnote{We take ${\rm i}\,z_3$ since this choice leads us, in the end, only to real matrices.}
 \begin{align}\label{x5}
 \mathcal{U}(x_1,x_2,t)=\begin{footnotesize}\begin{pmatrix}u_3(x_1,x_2,t)
 	\\
 	\vartheta_1(x_1,x_2,t)
 	\\\vartheta _2(x_1,x_2,t)
 	\end{pmatrix}\end{footnotesize}={\rm Re}\left[\begin{footnotesize}\begin{pmatrix} {\rm i}\,z_3(x_2)
 	\\
 	z_1(x_2)
 	\\z_2(x_2)
 \end{pmatrix}\end{footnotesize} e^{ {\rm i}\, k\, (  x_1-vt)}\right],
 \end{align}
 where  $v$ is the propagation speed {(the phase  velocity)}. If $z_i$, $i=1,2,3$, are solutions of the following systems
 \begin{align}
\begin{footnotesize}\begin{pmatrix} \mu_{\rm e} +\mu_{\rm c} &0	&0
 	\\
 	0& {\mu_{\rm e}\,L_{\rm c}^2}\,(\alpha_1+\alpha_2) 
 	&0
 	\\0
 	& 0 & {\mu_{\rm e}\,L_{\rm c}^2}\,(2\alpha_1+\alpha_3) \end{pmatrix}\end{footnotesize}\begin{footnotesize}\begin{pmatrix}	z_3''(x_2)
 	\\
 	z_1''(x_2)
 	\\\	z_2''(x_2)
 	\end{pmatrix}\end{footnotesize}+ {\rm i}\, \begin{footnotesize}\begin{pmatrix} 0& -2\mu_{\rm c} & 0  
 	\\
 	-2\mu_{\rm c} & 0 
 	& {\mu_{\rm e}\,L_{\rm c}^2}\,k\,(\alpha_1-\alpha_2+\alpha_3)
 	\\ 0  
 	& {\mu_{\rm e}\,L_{\rm c}^2}\,k\,(\alpha_1-\alpha_2+\alpha_3) & 0 \end{pmatrix}\end{footnotesize}\begin{footnotesize}\begin{pmatrix} 	z_3'(x_2)
 	\\
 	z_1'(x_2)
 	\\\	z_2'(x_2)
 	\end{pmatrix}\end{footnotesize}\notag\vspace{2mm}\\ -\begin{footnotesize}\begin{pmatrix} k^2\,(\mu_{\rm e} +\mu_{\rm c} )-\rho \,k^2v^2 & 0 & 2\mu_{\rm c}\,k
 	\\
 	0 & {\mu_{\rm e}\,L_{\rm c}^2}\,k^2\,(2\alpha_1+\alpha_3)-\rho j\,\mu_{\rm e}\,\tau_{\rm c}^2\, k^2v^2+4 \mu_{\rm c}
 	& 0
 	\\2\mu_{\rm c}\,k
 	& 0 & {\mu_{\rm e}\,L_{\rm c}^2}\,k^2\,(\alpha_1+\alpha_2)-\rho j\,\mu_{\rm e}\,\tau_{\rm c}^2\, k^2v^2+4 \mu_{\rm c} \end{pmatrix}\end{footnotesize}\begin{footnotesize}\begin{pmatrix} 	z_3(x_2)
 	\\
 	z_1(x_2)
 	\\	z_2(x_2)
 	\end{pmatrix}\end{footnotesize}=0,
 \end{align}
 and (from the boundary conditions)
 \begin{align}
 \begin{footnotesize}\begin{pmatrix} \mu_{\rm e} +\mu_{\rm c} &0	&0
 	\\
 	0& {\mu_{\rm e}\,L_{\rm c}^2}\,(\alpha_1+\alpha_2)   
 	&0
 	\\0
 	& 0 & {\mu_{\rm e}\,L_{\rm c}^2}\,(2\alpha_1+\alpha_3) \end{pmatrix}\end{footnotesize}\begin{footnotesize}\begin{pmatrix} 	z_3'(0)
 	\\
 	z_1'(0)
 	\\	z_2'(0)
 	\end{pmatrix}\end{footnotesize}+ {\rm i}\, \begin{footnotesize}\begin{pmatrix} 0& -2\,\mu_{\rm c} & 0  
 	\\
 	0 & 0 
 	& k{\mu_{\rm e}\,L_{\rm c}^2}\,(\alpha_1-\alpha_2)
 	\\0
 	& k{\mu_{\rm e}\,L_{\rm c}^2}\,\alpha_3 & 0  \end{pmatrix}\end{footnotesize}\begin{footnotesize}\begin{pmatrix}	z_3(0)
 	\\
 	z_1(0)
 	\\	z_2(0)
 \end{pmatrix}\end{footnotesize}=0,\notag
 \end{align}\normalsize
 where $\cdot '$ denotes the derivative with respect to $x_2$, then $\mathcal{U}$ given by the ansatz \eqref{x5} satisfies \eqref{x6} and \eqref{x17}.
In a more compact notation, the above equations admit the following equivalent form \begin{align}\label{11}
\frac{1}{k^2}\,{\mathbf{X}}\,z''(x_2)+ {\rm i}\,\frac{1}{k}\, ({\mathbf{Y}}+{\mathbf{Y}}^T)\,z'(x_2)-{\mathbf{Z}}\,z(x_2)+k^2\, v^2 \, \hat\id\,z(x_2)=&\,0,\\
\frac{1}{k^2}\,{\mathbf{X}}\,z'(0)+ {\rm i}\, \frac{1}{k}\,{\mathbf{Y}}^T\,z(0)=&\,0,\notag
\end{align} where
     the matrices  ${\mathbf{X}}\,,{\mathbf{Y}}\, , {\mathbf{Z}}$ and $\hat\id$  are defined by
 \begin{align}\label{x47}
 {\mathbf{X}}&=k^2\,\begin{footnotesize}\begin{pmatrix}  \mu_{\rm e} +\mu_{\rm c} &0	&0
 	\\
 	0& \mu_{\rm e}\,L_{\rm c}^2\,(\alpha_1+\alpha_2) 
 	&0
 	\\0
 	& 0 & {\mu_{\rm e}\,L_{\rm c}^2}\,(2\alpha_1+\alpha_3) \end{pmatrix}\end{footnotesize},\qquad\quad  {\mathbf{Y}}=k\,\begin{footnotesize}\begin{pmatrix}  0& 0 & 0
 	\\
 	-2\mu_{\rm c} & 0 
 	& k{\mu_{\rm e}\,L_{\rm c}^2}\,\alpha_3
 	\\0  
 	& k{\mu_{\rm e}\,L_{\rm c}^2}\,(\alpha_1-\alpha_2) & 0 \end{pmatrix}\end{footnotesize},\\
 {\mathbf{Z}}&=\begin{footnotesize}\begin{pmatrix} k^2\,(\mu_{\rm e} +\mu_{\rm c} ) & 0	& 2k\,\mu_{\rm c}
 	\\
 	0 & {\mu_{\rm e}\,L_{\rm c}^2}\,k^2\,(2\alpha_1+\alpha_3)+4 \mu_{\rm c}
 	& 0
 	\\2k\,\mu_{\rm c}
 	& 0 & {\mu_{\rm e}\,L_{\rm c}^2}\,k^2\,(\alpha_1+\alpha_2)+4 \mu_{\rm c} \, \end{pmatrix}\end{footnotesize}, \qquad  {
 \widehat{\id}=\begin{footnotesize}\begin{pmatrix} 
 	\varrho & 0 & 0 \\
 		0 & \rho j\,\mu_{\rm e}\,\tau_{\rm c}^2 & 0 \\
 		0 & 0 & \rho j\,\mu_{\rm e}\,\tau_{\rm c}^2\end{pmatrix}\end{footnotesize}}.\notag
 \end{align}
 The system \eqref{11} has a similar structure to that from classical linear elasticity \cite{fu2002new} but we still have to rewrite it in order to make it  manageable for our analysis.  In this respect, the  solution $z$ of \eqref{11} is equivalent to find a  solution $y$ of 
 \begin{align}\label{11t}
 \frac{1}{k^2}\,\widehat{\id}^{-1/2}\,{\mathbf{X}}\,\widehat{\id}^{-1/2}\,y''(x_2)+ {\rm i}\,\frac{1}{k}\, \widehat{\id}^{-1/2}\,({\mathbf{Y}}+{\mathbf{Y}}^T)\widehat{\id}^{-1/2}\,y'(x_2) -\widehat{\id}^{-1/2}\,{\mathbf{Z}}\,\widehat{\id}^{-1/2}\,y(x_2)+k^2\, v^2 \, \id\,y(x_2)&=\,0,\\
  \frac{1}{k^2}\,\widehat{\id}^{-1/2}\,{\mathbf{X}}\,\widehat{\id}^{-1/2}\,y'(0)+ {\rm i}\, \frac{1}{k}\,\widehat{\id}^{-1/2}\,{\mathbf{Y}}^T\,\widehat{\id}^{-1/2}\,y(0)&=\,0,\notag
 \end{align}
 where  $y(x_2):=\widehat{\id}^{1/2}\,z(x_2)$. Therefore we  {may} use the modified matrices
 \begin{align}\label{nTQ}
 \boldsymbol{\mathcal{X}}:=&\,k^2\,\begin{footnotesize}\begin{pmatrix}  \frac{\mu_{\rm e} +\mu_{\rm c}}{\rho} &0	&0
 	\\
 	0& \frac{L_{\rm c}^2\,(\alpha_1+\alpha_2)}{\rho j\,\tau_{\rm c}^2} 
 	&0
 	\\0
 	& 0 & \frac{{L_{\rm c}^2}\,(2\alpha_1+\alpha_3)}{\rho j\,\tau_{\rm c}^2} \end{pmatrix}\end{footnotesize}, \qquad \qquad 
\boldsymbol{\mathcal{Y}}:= k\,\begin{footnotesize}\begin{pmatrix}  0& 0 & 0
 	\\
 	\frac{-2\mu_{\rm c}}{{\rho\,\tau_{\rm c}\sqrt{j\,\mu_{\rm e}}}} & 0 
 	& \frac{k{\,L_{\rm c}^2}\,\alpha_3}{\rho j\,\tau_{\rm c}^2}
 	\\0  
 	& \frac{k{\,L_{\rm c}^2}\,(\alpha_1-\alpha_2)}{\rho j\,\tau_{\rm c}^2} & 0 \end{pmatrix}\end{footnotesize},\notag\\
\boldsymbol{\mathcal{Z}}:=&\begin{footnotesize}\begin{pmatrix} \frac{k^2\,(\mu_{\rm e} +\mu_{\rm c} )}{\rho} & 0	& \frac{2k\,\mu_{\rm c}}{{\rho\,\tau_{\rm c}\sqrt{j\,\mu_{\rm e}}}}
 	\\
 	0 & \frac{{\mu_{\rm e}\,L_{\rm c}^2}\,k^2\,(2\alpha_1+\alpha_3)+4 \mu_{\rm c}}{\rho j\,\mu_{\rm e}\,\tau_{\rm c}^2}
 	& 0
 	\\\frac{2k\,\mu_{\rm c}}{{\rho\,\tau_{\rm c}\sqrt{j\,\mu_{\rm e}}}}
 	& 0 & \frac{{\mu_{\rm e}\,L_{\rm c}^2}\,k^2\,(\alpha_1+\alpha_2)+4 \mu_{\rm c}}{\rho j\,\mu_{\rm e}\,\tau_{\rm c}^2} \, \end{pmatrix}\end{footnotesize},
 \end{align}
 and the following equivalent form of the system \eqref{11} emerges
 \begin{align}\label{n11}
 \frac{1}{k^2}\boldsymbol{\mathcal{X}}\,y''(x_2)+ {\rm i}\,\frac{1}{k} (\boldsymbol{\mathcal{Y}}+\boldsymbol{\mathcal{Y}}^T)\,y'(x_2)-\boldsymbol{\mathcal{Z}}\,y(x_2)+k^2\, v^2 \, \id\,y(x_2)=&\,0,\\
 \frac{1}{k^2}\boldsymbol{\mathcal{X}}\,y'(0)+ {\rm i}\,\frac{1}{k} \,\boldsymbol{\mathcal{Y}}^T\,y(0)=&\,0.\notag
 \end{align}

 \begin{lemma}If  the constitutive coefficients satisfy the conditions 	\eqref{d12},
 then the	 matrices $\boldsymbol{\mathcal{Z}}$ and $\boldsymbol{\mathcal{X}}$ are symmetric and positive definite.	
 \end{lemma}
 \begin{proof}
 	Symmetry is clear and it is easy to compute that	\begin{align}
 	{\mathbf{X}}_{11}&= k^2(\mu_{\rm e} +\mu_{\rm c}) \,,\qquad 
 	{\mathbf{X}}_{11}{\mathbf{X}}_{22}-{\mathbf{X}}_{12}{\mathbf{X}}_{21}=k^4\,\mu_{\rm e}\,L_{\rm c}^2\,( \mu_{\rm e} +\mu_{\rm c} )(\alpha_1+\alpha_2)\,,\\ 
 	\det {{\mathbf{X}}}&=k^6\,\mu_{\rm e}^2\,L_{\rm c}^4\,( \mu_{\rm e} +\mu_{\rm c} )(\alpha_1+\alpha_2)(2\alpha_1+\alpha_3).\notag
 	\end{align}
 	Therefore, our constitutive  hypothesis imply that ${\mathbf{X}}$ is positive-definite. In addition,  ${\mathbf{Z}}$ is positive-definite if and only if the principal minors are positive, namely
 	\begin{align}
 	{\mathbf{Z}}_{11}&=k^2(\mu_{\rm e} +\mu_{\rm c}),\qquad
 	{\mathbf{Z}}_{11}{\mathbf{Z}}_{22}-{\mathbf{Z}}_{12}{\mathbf{Z}}_{21}=k^2(\mu_{\rm e} +\mu_{\rm c})\left[{\mu_{\rm e}\,L_{\rm c}^2}\,k^2\,(2\alpha_1+\alpha_3)+4 \mu_{\rm c}\right]\,,\\
 	\det({\mathbf{Z}})&=\mu_{\rm e}\,k^2\left[\mu_{\rm e}\,L_{\rm c}^2\,k^2\,(2\alpha_1+\alpha_3)+4 \mu_{\rm c}\right]\left[L_{\rm c}^2\,k^2(\mu_{\rm e} +\mu_{\rm c})(\alpha_1+\alpha_2)+4 \mu_{\rm c}\right]\,,\notag
 	\end{align}
 i.e. under the hypothesis of the lemma. 
 Since $\mathbf{X}$ and $\mathbf{Z}$ are positive definite, so there are $\boldsymbol{\mathcal{X}}$ and $\boldsymbol{\mathcal{Z}}$ defined by \eqref{nTQ}, and the proof is complete.
 \end{proof}
 \subsection{The limiting speed}

We now  seek  a solution $y$ of the differential system \eqref{n11} in the form 
 \begin{align}\label{x7}
 y(x_2)=\begin{footnotesize}\begin{pmatrix} 	d_3
 	\\
 	d_1
 	\\d_2
 	\end{pmatrix}\end{footnotesize} \,e^{{\rm i}\,r\,k\,x_2},\qquad \text{Im}\,r>0,
 \end{align}
 where $r \, \in \mathbb{C}$ is a complex parameter, $d=\begin{footnotesize}\begin{pmatrix} 	d_3,
 &
 d_1,
 &d_2
 \end{pmatrix}\end{footnotesize}^T \in \mathbb{C}^3 $, $d
 \neq 0$ is the amplitude and $ \text{Im}\,r$ is the coefficient of the imaginary part of $r$. The  condition $\text{Im}\,r>0$ ensures the asymptotic decay condition \eqref{04}. 
 Inserting $\eqref{x7}$ in $\eqref{n11}_1$ we obtain the  systems of algebraic equations
 \begin{align}\label{x8}
 [r^2\boldsymbol{\mathcal{X}}+r\,(\boldsymbol{\mathcal{Y}}+\boldsymbol{\mathcal{Y}}^T)+\boldsymbol{\mathcal{Z}}-\ k^2 v^2 {\id}]\,d=0,\qquad 
 [r\,\boldsymbol{\mathcal{X}}+\boldsymbol{\mathcal{Y}}^T]\,d=0.
 \end{align}
 The characteristic equation corresponding to the eigenvalue problem \eqref{x8}$_1$, i.e., the condition to have a nontrivial solution of $d=\begin{footnotesize}\begin{pmatrix} 	d_3,
&
 d_1,
&d_2
 \end{pmatrix}\end{footnotesize}^T$, is 
 \begin{align}\label{x9}
 \det\,[r^2\boldsymbol{\mathcal{X}}+r(\boldsymbol{\mathcal{Y}}+\boldsymbol{\mathcal{Y}}^T)+\boldsymbol{\mathcal{Z}}-k^2 v^2 {\id}]=0,
 \end{align}
 which gives six roots as eigenvalue $r$. The associated eigenvectors $d$ can be determined for the corresponding eigenvalues. 
\begin{definition}
	By the {\bf limiting speed} we understand a speed $\widehat{v}>0$, such that for all wave speeds satisfying 	$0\le v<\widehat{v}$  (subsonic speeds) the roots  of the characteristic equation \eqref{x9} are not real\footnote{{Since the quadratic equation does not have real solutions, there is a complex solution $r$ for which ${\rm Im}\, r>0$, since the complex solutions are pair-conjugated. Therefore, the existence of such a solution $r$, i.e., ${\rm Im}\, r>0$, implies the existence of a  wave propagating in the direction $x_1$ with the
		phase velocity $v$ and decaying exponentially in the direction $x_2$.}} and vice versa, i.e., if the roots  of the characteristic equation \eqref{x9} are not real then they correspond to wave speeds $v$ satisfying 	$0\le v<\widehat{v}$.
\end{definition} 
 \begin{proposition}\label{lemmaGH} Let  {us} assume that  the constitutive coefficients satisfy the conditions 	\eqref{d12}. The following  {holds} true:
 	\begin{enumerate}
 \item	There	exists a limiting speed $\widehat{v}>0$. Furthermore,   if one root   $r_v$ of the characteristic equation \eqref{x9} is real then it corresponds to a speed $v\geq \widehat{v}$ (non-admissible).
 \item For all $\theta\in \left(-\frac{\pi}{2},\frac{\pi}{2}\right)$ and $k>0$, the tensor $\boldsymbol{\mathcal{Z}}_\theta:=\sin^2\theta\boldsymbol{\mathcal{X}}+\sin\theta\cos\theta (\boldsymbol{\mathcal{Y}}+\boldsymbol{\mathcal{Y}}^T)+\cos^2\theta\boldsymbol{\mathcal{Z}}$ is positive definite.
 \item For all $\theta\in \left(-\frac{\pi}{2},\frac{\pi}{2}\right)$, $k>0$ and $0\le v<\widehat{v}$, the tensor $\widetilde{\boldsymbol{\mathcal{Z}}}_\theta:=\sin^2\theta\,\boldsymbol{\mathcal{X}}+\sin\theta\cos\theta (\boldsymbol{\mathcal{Y}}+\boldsymbol{\mathcal{Y}}^T)+\cos^2\theta\,\boldsymbol{\mathcal{Z}}-k^2 v^2 \cos^2\theta \,{\id}$ is positive definite.
 \end{enumerate}
 \end{proposition}
 \begin{proof} The proof is quite similar to the equivalent proposition concerning the Rayleigh wave, see \cite{khan2022existence}. However, since there are different calculations, we write explicitly some key points.
 	\begin{enumerate}
 		\item 
Assume that there exists a real $r_v$ as solution  of the characteristic equation \eqref{x9}, then $ \exists \,\theta \in (-\frac{\pi}{2},\frac{\pi}{2})$ such that $r_v=\tan \theta$. Therefore, corresponding to \eqref{x7},  $\mathcal{U}$  given by \eqref{x5}  and defined by $r_v$	turns into
 \begin{align}\label{46}
 \mathcal{U}(x_1,x_2,t)=\begin{footnotesize}
\begin{pmatrix} u_3(x_1,x_2,t)
 	\\
 	\vartheta _1(x_1,x_2,t)
 	\\\vartheta _2(x_1,x_2,t)
 	\end{pmatrix}\end{footnotesize}&=
\begin{footnotesize}
\begin{pmatrix} {\rm i}\, \,d_3 \\
	d_1
 	\\
 	d_2
 	\end{pmatrix}\end{footnotesize}e^{ik\, (  x_1+\tan\theta x_2-vt)}=\begin{footnotesize}
 		\begin{pmatrix} {\rm i}\, \,d_3 \\
	d_1
 	\\
 	d_2
 	\end{pmatrix}\end{footnotesize}e^{\frac{ik}{\cos\theta}(\cos\theta x_1+\sin\theta x_2-\cos\theta vt)},
 \end{align}
 which means that $\mathcal{U}(x_1,x_2,t)$ is   a non-trivial  plane body wave solution with  wave number $\frac{k}{\cos\theta}$, the speed $v_\theta=v\cos\theta$ and propagation in the direction ${n}_\theta$ where 
 $	{n}_\theta=(\cos\theta,\sin\theta,0)$. A direct substitution of \eqref{46} into \eqref{PDE} together with the notation  $\begin{footnotesize}\begin{pmatrix}	f_3
 		\\
 		f_1
 		\\	f_2
 \end{pmatrix}\end{footnotesize}=\widehat{\id}^{1/2}\begin{footnotesize}\begin{pmatrix}	d_3
 		\\
 		d_1
 		\\	d_2
 \end{pmatrix}\end{footnotesize}$ 
 imply the existence of a non-trivial solution  $\begin{footnotesize}\begin{pmatrix}f_1,&f_2,&f_3\end{pmatrix}\end{footnotesize}\neq 0$ of the algebraic system written in matrix format
\begin{align}\label{x093}
\left[	\sin^2\theta\,\boldsymbol{\mathcal{X}}+\sin\theta\cos\theta (\boldsymbol{\mathcal{Y}}+\boldsymbol{\mathcal{Y}}^T)+\cos^2\theta\,\boldsymbol{\mathcal{Z}}-k^2 v^2_\theta \,{\id}\right]\,\begin{footnotesize}\begin{pmatrix} 	f_3
\\
f_1
\\	f_2
\end{pmatrix}\end{footnotesize}=0.
\end{align}

Let us observe that  equation \eqref{x093}  is actually the propagation condition for plane waves in isotropic Cosserat materials, in the fixed direction $	{n}_\theta=(\cos\theta,\sin\theta,0)$. Since  the constitutive coefficients satisfy the conditions 	\eqref{d12}, according to Proposition \ref{proprealw2}, for the direction $	{n}_\theta=(\cos\theta,\sin\theta,0)$ in particular, the system of partial differential equations \eqref{PDE} admits a non trivial solution in the form
\begin{align}\label{ansatzwp}
u(x_1,x_2,t)&={\rm i}\,\begin{footnotesize}\begin{pmatrix}0\\0\\\widehat{u}_3\end{pmatrix}\end{footnotesize}
\, e^{{\rm i}\, \left(k\langle {\xi},\, x\rangle_{\mathbb{R}^3}-\,\omega \,t\right)}\,,\qquad\qquad  \vartheta (x_1,x_2,t)=\begin{footnotesize}\begin{pmatrix}\widehat{\vartheta }_1\\\widehat{\vartheta }_2\\0\end{pmatrix}\end{footnotesize} \, e^{{\rm i}\, \left(k\langle {\xi},\, x\rangle_{\mathbb{R}^3}-\,\omega \,t\right)},\\&
(\widehat{u}_3,\widehat{\vartheta }_1, \widehat{ \vartheta_2   })^T\in\mathbb{C}^{3}, \quad (\widehat{u}_3,\widehat{\vartheta }_1, \widehat{ \vartheta_2   })^T\neq 0\,\notag
\end{align}
only for real positive values $\omega^2$.
To each  $\theta \in(-\frac{\pi}{2},\frac{\pi}{2})$  we associate these  real frequencies $\omega_\theta$ satisfying 
\begin{align}
\det\,\{	\sin^2\theta\boldsymbol{\mathcal{X}}+\sin\theta\cos\theta (\boldsymbol{\mathcal{Y}}+\boldsymbol{\mathcal{Y}}^T)+\cos^2\theta\boldsymbol{\mathcal{Z}}-\omega^2_\theta \,{\id}\}=0.
\end{align}
Then, each $\omega_\theta$ defines a $v_\theta$ such that $ \omega_\theta=v_\theta\cos\theta$. We define $\widehat{v}$ as the minimum of the values of $v_\theta\cos\theta$ for all $\theta\in(-\frac{\pi}{2},\frac{\pi}{2})$. Hence,  this means that after  {we} know all  solutions $v_\theta$ of 
 \begin{align}\label{x11}
 \det\,\{	\sin^2\theta\boldsymbol{\mathcal{X}}+\sin\theta\cos\theta (\boldsymbol{\mathcal{Y}}+\boldsymbol{\mathcal{Y}}^T)+\cos^2\theta\boldsymbol{\mathcal{Z}}-k^2 v^2_\theta \cos^2\theta \,{\id}\}=0,
 \end{align}
 we  {can} define 
 \begin{align}
 \widehat{v}=\inf_{\theta\in(-\frac{\pi}{2},\frac{\pi}{2})} v_\theta.\label{lims}
 \end{align}
 
 In conclusion,  if  there exists a value of $v$ such that the equation \eqref{x9} admits a real solution $r_v$, then $v$  must satisfy 
 $
 v\geq \widehat{v}.
$
Thus,  if $v$ is such that
 $
 0\le v<\widehat{v},
 $
 then $r$ can not be real and if $r$ is real, then $v\geq\widehat{v}$.
\item 
	Since  the constitutive coefficients satisfy the conditions 	\eqref{d12}, according to Proposition \ref{proprealw2}, it follows that there exists \textit{only} real numbers $\omega$ such that the following  system\footnote{This system follows directly by inserting  \eqref{ansatzwp} into \eqref{PDE} and obtaining that   $\widehat{u}_3, \widehat{\vartheta}_1$ and $\widehat{\vartheta}_2$ have to satisfy the following partial differential equations
		\begin{align*}
		-\omega^2\rho\,\widehat{u}_3&=-\,k^2\,(\mu_{\rm e} +\mu_{\rm c})\dd \widehat{u}_3 \sum_{l=1}^2 \xi_l^2-2k\mu_{\rm c}\,\xi_l\,\widehat{\vartheta }_2 +2k\mu_{\rm c}\,\xi_2\,\widehat{\vartheta }_1 ,
		\\
		\omega^2\rho\, j\,\mu_{\rm e}\,\tau_{\rm c}^2\,\widehat{ \vartheta }_1&=k^2\,{\mu_{\rm e}L_{\rm c}^2}\,(2\alpha_1+\alpha_3)\widehat{ \vartheta }_1\xi_1^2+k^2\,{\mu_{\rm e}L_{\rm c}^2}\,(\alpha_1+\alpha_2)\widehat{ \vartheta }_1\xi_2^2+k^2{\mu_{\rm e}L_{\rm c}^2}\,(\alpha_1-\alpha_2+\alpha_3)\widehat{ \vartheta }_2\xi_1\xi_2 -2\,k\,\mu_{\rm c}\widehat{u}_3 \xi_2+4\, \mu_{\rm c} \widehat{\vartheta }_1,\vspace{1mm}\notag\\
		\omega^2\rho\, j\,\mu_{\rm e}\,\tau_{\rm c}^2\,\widehat{ \vartheta }_2&=k^2\,{\mu_{\rm e}L_{\rm c}^2}\,(\alpha_1+\alpha_2)\widehat{ \vartheta }_2\xi_1^2+k^2\,{\mu_{\rm e}L_{\rm c}^2}\,(2\alpha_1+\alpha_3)\widehat{ \vartheta }_2\xi_2^2+k^2{\mu_{\rm e}L_{\rm c}^2}\,(\alpha_1-\alpha_2+\alpha_3)\widehat{ \vartheta }_1\xi_1\xi_2 +2\,k\,\mu_{\rm c}\widehat{u}_3 \xi_1+4\, \mu_{\rm c} \widehat{\vartheta }_2.\notag\\
		\end{align*}}  admits non trivial solutions $w=\begin{footnotesize}\begin{pmatrix}
	\widehat{u}_3, \widehat{\vartheta}_1, \widehat{\vartheta}_2
	\end{pmatrix}\end{footnotesize}^T\neq 0$  for  real positive numbers $\omega^2$, i.e.,
	\begin{align}\label{vz1}
	[\widetilde{\mathbf{Q}}_2(\xi,k)-\omega^2\widehat{\id}]\, w=0,
	\end{align}
		where  
	\begin{align}\label{Q_1}
		\mathbf{\mathbf{Q}}_2(\xi,k)&=\begin{footnotesize}\begin{footnotesize}\begin{pmatrix}
					k^2(\mu_{\rm e} +\mu_{\rm c})& -2\mu_{\rm c}\,k\,\xi_2 & 2\mu_{\rm c}\,k\,\xi_1\vspace{2mm}\\
					-2\mu_{\rm c}\,k\,\xi_2 & {\mu_{\rm e}\,L_{\rm c}^2}\,k^2\left[(2\alpha_1+\alpha_3)\xi_1^2+(\alpha_1+\alpha_2)\xi_2^2\right]+4 \mu_{\rm c} & {\mu_{\rm e}\,L_{\rm c}^2}\,k^2\,(\alpha_1-\alpha_2+\alpha_3)\xi_1\xi_2\vspace{2mm}\\
					2\mu_{\rm c}\,k\,\xi_1 & {\mu_{\rm e}\,L_{\rm c}^2}\,k^2\,(\alpha_1-\alpha_2+\alpha_3)\xi_1\xi_2 & {\mu_{\rm e}\,L_{\rm c}^2}\,k^2\left[(\alpha_1+\alpha_2)\xi_1^2+(2\alpha_1+\alpha_3)\xi_2^2\right]+4 \mu_{\rm c}
		\end{pmatrix}\end{footnotesize}\end{footnotesize}
	\end{align}

	But this is equivalent to  the positive definiteness of the matrix
	$
		\widetilde{\mathbf{\mathbf{Q}}}_2(\xi,k)
		=\widehat{\id}^{-1/2}	\mathbf{\mathbf{Q}}_2(\xi,k)\widehat{\id}^{-1/2}.
$
	Since the conditions \eqref{d12} imply the positive definiteness of the matrix $\widetilde{\mathbf{\mathbf{Q}}}_1(\xi,k)$ for all $\xi=(\xi_1,\xi_2,0)\in \mathbb{R}^3$, $\lVert \xi\rVert=1$, taking $\xi=(\cos\,\theta, \sin\,\theta,0)$ we deduce  the desired result.
\item 

Since $\boldsymbol{\mathcal{Z}}_\theta:=\sin^2\theta\boldsymbol{\mathcal{X}}+\sin\theta\cos\theta (\boldsymbol{\mathcal{Y}}+\boldsymbol{\mathcal{Y}}^T)+\cos^2\theta\boldsymbol{\mathcal{Z}}$ is positive definite, it admits only positive eigenvalues.   Assuming that there exist  $\theta_0\in \left(-\frac{\pi}{2},\frac{\pi}{2}\right)$ and  $v_0\in[0,\widehat{v})$ for which there is an eigenvalue $\lambda_{\theta_0}$   of $\boldsymbol{\mathcal{Z}}_\theta$
such that $\lambda_{\theta_0}<k^2 v^2_0 \cos^2\theta_0 $, then 
\begin{align}
v_{\theta_0}:=\sqrt{\frac{\lambda_{\theta_0}}{k^2\,\cos^2\theta_0}}<v_0<\widehat{v}
\end{align}
is solution of \eqref{x11}, i.e., for fixed $\theta_0$ we have that $v_{\theta_0}<\widehat{v}$  verifies
 \begin{align}\label{x11n}
\det\,\{	\sin^2{\theta_0}\,\boldsymbol{\mathcal{X}}+\sin{\theta_0}\,\cos{\theta_0} \,(\boldsymbol{\mathcal{Y}}+\boldsymbol{\mathcal{Y}}^T)+\cos^2{\theta_0}\,\boldsymbol{\mathcal{Z}}-k^2 v^2_{{\theta_0}}\, \cos^2{\theta_0} \,{\id}\}=0.
\end{align} 
This is in contradiction to the definition of the limiting speed and Proposition \ref{lims}, since $\widehat{v}$ is the smallest speed having this property. Therefore, it remains that for all  $\theta\in \left(-\frac{\pi}{2},\frac{\pi}{2}\right)$, $k>0$ and for all $0\le v<\widehat{v}$, all the eigenvalues of $\boldsymbol{\mathcal{Z}}_\theta$ are larger than $k^2 v^2 \cos^2\theta$ and the proof is complete.\qedhere
	\end{enumerate}
\end{proof}

 \subsection{Derivation and matrix analysis of the algebraic Riccati equation}
 The main ingredient of the method used by Fu and Mielke \cite{fu2002new} is to  look at \eqref{n11} as  an initial value problem and to search for a solution in the form 
 \begin{align}\label{08}
 y(x_2)=e^{-k\,x_2 \,\boldsymbol{\mathcal{E}}}y(0),
 \end{align}
 where $\boldsymbol{\mathcal{E}}\in\mathbb{C}^{3 \times 3}$  is to be determined. On substituting \eqref{08} into  \eqref{x8}, we get
 \begin{align}\label{09}
 [\boldsymbol{\mathcal{X}}\boldsymbol{\mathcal{E}}^2- {\rm i}\, (\boldsymbol{\mathcal{Y}}+\boldsymbol{\mathcal{Y}}^T)\boldsymbol{\mathcal{E}}-\boldsymbol{\mathcal{Z}}+\ k^2v^2 {\id}]\, y(x_2)=0, \qquad\qquad \qquad  (-\boldsymbol{\mathcal{X}}\boldsymbol{\mathcal{E}}+ {\rm i}\, \,\boldsymbol{\mathcal{Y}}^T)\,y(0)=0.
 \end{align}
 It is clear that in order to have a proper decay the eigenvalues of $\boldsymbol{\mathcal{E}}$ have to be such that their real part is positive. 
 
 For the linear Cosserat model,  we introduce the so called  \textbf{surface impedance matrix} 
 \begin{align}\label{10}
\boldsymbol{\boldsymbol{\mathcal{M}}}=-(-	\boldsymbol{\mathcal{X}}\boldsymbol{\mathcal{E}}+ {\rm i}\,\boldsymbol{\mathcal{Y}}^T)\qquad \iff \qquad \boldsymbol{\mathcal{E}}=\boldsymbol{\mathcal{X}}^{-1}(\boldsymbol{\mathcal{M}}+ {\rm i}\, \,\boldsymbol{\mathcal{Y}}^T).
 \end{align}
 The reason to consider  this replacement of $\boldsymbol{\mathcal{E}}$ to $\boldsymbol{\mathcal{M}}$ is to convert the equation $\eqref{09}_{2}$ into an  equation for a Hermitian matrix $\boldsymbol{\mathcal{M}}$. On substituting $\eqref{10}_2$ into $\eqref{09}_1$, we obtain
 \begin{align}\label{12}
 (\boldsymbol{\mathcal{M}}- {\rm i}\,\boldsymbol{\mathcal{Y}})\boldsymbol{\mathcal{X}}^{-1}(\boldsymbol{\mathcal{M}}+ {\rm i}\, \mathcal{R^T})-\boldsymbol{\mathcal{Z}}+k^2\, v^2\,{\id}=0, \qquad\boldsymbol{\boldsymbol{\mathcal{M}}}\, y(0)=0.
 \end{align}
 
  We call equation $\eqref{12}_1$  \textbf{algebraic Riccati equation} for the propagation of Love waves in  {the} linear Cosserat model.  
  
   From this point  {on}, the entire problem is reduced to a purely mathematical problem. Moreover, due to the properties of the known matrices involved in the Riccati equation, as well as the definition of the limiting speed, we are exactly in the framework of the mathematical results obtained by Fu and Mielke \cite{mielke2004uniqueness}. 
 Since we are interested in a nontrivial solution  $y$, we impose $y(0)
\neq0$, so that the matrix $\boldsymbol{\mathcal{M}}$ has to satisfy
 \begin{align}\label{x16}
 {\rm det}\,\boldsymbol{\mathcal{M}}=0.
 \end{align}
 	The  equation \eqref{x16} is called  \textbf{secular equation} for the {\bf Love waves propagation} in linear Cosserat model in terms of the \textbf{impedance matrix} $\boldsymbol{\mathcal{M}}$.

 The unknown  {quantities} are  {now} $v$ and $\boldsymbol{\mathcal{M}}$ solutions of the equations \eqref{x16} and  {\eqref{12}}. These equations are coupled and nonlinear.  On one hand, there is not only one   Hermitian matrix $\boldsymbol{\mathcal{M}}$ satisfying \eqref{12}. On the other hand, it is not clear if there exists  {a unique} solution $v$ of \eqref{x16}, even if all Hermitian matrix $\boldsymbol{\mathcal{M}}$ satisfying \eqref{12} are known. The strategy which works is to look at the equations \eqref{12} and to remark that  equation \eqref{12}$_1$ leads to a mapping $v\mapsto\boldsymbol{\boldsymbol{\mathcal{M}}}_v$, where $\boldsymbol{\mathcal{M}}_v$ is the solution of \eqref{12}$_1$ for a fixed $v$. Then equation \eqref{x16}  becomes the equation which determines the wave speed  $v$.  However, first we need to identify the domain of those $v$ for which  \eqref{12}$_1$ admits a unique solution. Then we  prove the existence and uniqueness of the solution $v$ of the secular equation. But this solution $v$ has to be such that the  corresponding   solution $\boldsymbol{\mathcal{M}}_v$ of the Riccati equation \eqref{12} assures  that $\text{Re(spec}\,\boldsymbol{\mathcal{E}}{\rm )}$ is positive, 	where ``$\text{\rm Re\,(spec\,}\boldsymbol{\mathcal{E}}{\rm )}$" means the ``real part of spectra of $\boldsymbol{\mathcal{E}}$".

 Everything we have done until now in the paper  {shows} that we have put the problem in the framework of   the general result obtained by  Fu and Mielke in \cite{fu2001nonlinear,fu2002new} (see also the work by Mielke  and Sprenger \cite{mielke1998quasiconvexity}) and it follows  {directly}:
 \begin{theorem}{\rm {\bf [The main result of this paper]}}\label{fmth} 
 	\begin{enumerate}
 		\item 	If $0\leq v<\widehat{v}$, then the matrix problem 
 		\begin{align}\label{13}
 			\boldsymbol{\mathcal{X}}{\mathcal{E}}^2- {\rm i}\, (  \boldsymbol{\mathcal{Y}}+   \boldsymbol{\mathcal{Y}}^T)\mathcal{E}-  \boldsymbol {\mathcal{Z}}+ k^2 v^2{\id}=0,\qquad \textrm{\rm Re\,spec\,}  \boldsymbol{\mathcal{E}}>0,
 		\end{align}
 		where $\textrm{\rm Re\,spec\,}  \boldsymbol{\mathcal{E}} $  means the real part of the  spectra of $  \boldsymbol{\mathcal{E}}$,  has a unique solution for $  \boldsymbol{\mathcal{E}}$ and $  {  \boldsymbol{\mathcal{M}}}:=-(-	\boldsymbol{\mathcal{X}}  \boldsymbol{\mathcal{E}}+ {\rm i}\,  \boldsymbol{\mathcal{Y}}^T)$ is hermitian.

 		\item 
 		If $0\leq v<\widehat{v}$, then the unique solution $\mathcal{M}$ of the \textit{algebraic Riccati equation} 
 		\begin{align}\label{120}
 			(  \boldsymbol{\mathcal{M}}- {\rm i}\,  \boldsymbol{\mathcal{Y}})  \boldsymbol{\mathcal{X}}^{-1}(  \boldsymbol{\mathcal{M}}+ {\rm i}\, \boldsymbol{\mathcal{Y}}^T)- \boldsymbol {\mathcal{Z}}+k^2\, v^2\,{\id}=0, \qquad  {  \boldsymbol{\mathcal{M}}}\, y(0)=0.
 		\end{align}
 		that satisfies  $\text{\rm Re\,spec}\,(   \boldsymbol{\mathcal{X}}^{-1}( \boldsymbol {\mathcal{M}}+{\rm i}\,  \boldsymbol{\mathcal{Y}}^T))>0$ is given explicitly by
 		\begin{align}\label{fmv}
 			\boldsymbol{\mathcal{M}}=\Big(\int_{0}^{\pi}	  \boldsymbol {\mathcal{X}}_  { \theta} ^{-1}\, d\theta\Big)^{-1}\Big(\pi\id-  {\rm i}\, \int_{0}^{\pi}   \boldsymbol{\mathcal{X}}_  { \theta} ^{-1}   {   \boldsymbol{\mathcal{Y}}}_\theta  ^T\, d\theta\Big),
 		\end{align}
 			where the  matrices $ \boldsymbol {\mathcal{X}}_  { \theta}$, $  \boldsymbol{\mathcal{Y}}_  { \theta}$ and $  \boldsymbol{\mathcal{Z}}_  { \theta}$ are  obtained by rotation of the old coordinate system
 		\begin{align}\label{21}
 			\boldsymbol{\mathcal{X}}_  { \theta} & = \cos^2\theta\,  \boldsymbol{\mathcal{X}}-\sin\theta\cos\theta\,(  \boldsymbol{\mathcal{Y}}+  \boldsymbol{\mathcal{Y}}^T)+\sin^2\theta\,\widetilde{  \boldsymbol{\mathcal{Z}}},\notag\vspace{2mm}\\
 			\boldsymbol{\mathcal{Y}}_\theta  & = \cos^2\theta\,  \boldsymbol{\mathcal{Y}}+\sin\theta\cos\theta
 			\,(  \boldsymbol{\mathcal{X}}-\widetilde{  \boldsymbol{\mathcal{Z}}})-\sin^2\theta\,  \boldsymbol{\mathcal{Y}}^T,\vspace{2mm}\\
 			\widetilde{  \boldsymbol{\mathcal{Z}}}_\theta & = \cos^2\theta\,\widetilde{  \boldsymbol{\mathcal{Z}}}+\sin\theta\cos\theta\,(  \boldsymbol{\mathcal{Y}}+  \boldsymbol{\mathcal{Y}}^T)+\sin^2\theta\,  \boldsymbol{\mathcal{X}}\notag
 		\end{align}
 			with   $\widetilde{  \boldsymbol{\mathcal{Z}}}=  \boldsymbol{\mathcal{Z}}-k^2\,v^2\,{\id}$.
 		\item 	If $0\leq v<\widehat{v}$, then  the solution  $  {\mathcal{M}}_v$ obtained from \eqref{120} has the following properties
 		\begin{enumerate}
 			\item $  \boldsymbol{\mathcal{M}}_v$ is hermitian,
 			\item  $\frac{d  {  \boldsymbol{\mathcal{M}}}_v}{d v}$ is negative definite,
 			\item  $\tr(  {  \boldsymbol{\mathcal{M}}}_v)\geq 0$, and $\langle  w,  {  \boldsymbol{\mathcal{M}}}_v\,w\rangle \geq 0$ for all real vectors $w$ for  all $0\leq v\leq  \widehat{v}$,
 			\item $  \boldsymbol{\mathcal{M}}_v$ is  positive definite.
 		\end{enumerate}
 		\item The secular equation
 		\begin{align}\label{secularEq}
 			\det  {  \boldsymbol{\mathcal{M}}}_v=0,
 		\end{align}
 		where $  \boldsymbol{\mathcal{M}}_v$ is obtained from $\eqref{120}$, has a unique admissible solution $0\leq v<\widehat{v}$. In other words there {\bf exists a unique Love wave propagating in the Cosserat medium}.
 	\end{enumerate}
 \end{theorem}

 \begin{proof}
 	The reader may consult the paper by Fu and Mielke \cite{fu2002new}  since we have transformed the Love wave propagation problem in the framework  {so that} the  Fu and Mielke's results could be directly applied. Complete explanations in our notations may be found in \cite{khan2022existence}, too.
 \end{proof}
 
 From a computational perspective, the key advantage of the aforementioned result lies in the fact that, since $\boldsymbol{\mathcal{X}}_{\theta}$ and $\boldsymbol{\mathcal{Y}}_{\theta}$ depend on the wave speed $v$, the secular equation $\det \boldsymbol{\mathcal{M}}_v=0$ can be formulated explicitly without requiring prior knowledge of the analytical expressions (as functions of the wave speed) for the eigenvalues satisfying \eqref{x9} and the associated eigenvector $d^{(k)}$. This bypasses the primary challenge encountered in most generalised  {continuum} models, except for specific cases where the task is relatively straightforward, such as classical isotropic linear elasticity \cite{Achenbach} or the theory of materials with voids \cite{NunziatoCowin79,ChiritaGhiba2,straughan2008stability}, provided more restrictive conditions are imposed on the constitutive coefficients.

\section{Numerical implementation}\label{Num1}\setcounter{equation}{0}

In what follows we will do the calculations for dense polyurethane  {Foam@0.18mm} with the constitutive parameters given in \cite{Lakes83,lakes1987foam}:
\begin{align}\label{numcoef1}
\mu_{\rm e}=104 \,\text{(MPa)},\qquad \mu_{\rm c}=4.3331  \,\text{(MPa)}, \qquad {\mu_{\rm e}\,L_{\rm c}^2}\,\alpha_1=79.9552 \,(\text{MPa}\cdot \text{mm}^2), \notag\vspace{2mm}\\ {\mu_{\rm e}\,L_{\rm c}^2}\,\alpha_2=10.6496 \,(\text{MPa}\cdot \text{mm}^2), \qquad {\mu_{\rm e}\,L_{\rm c}^2}\,\alpha_3= -53.3035 \,(\text{MPa}\cdot \text{mm}^2), \vspace{2mm}\\
J=j\,\mu_{\rm e}\,\tau_{\rm c}^2= 10 \,(\text{mm}^2),\qquad \rho_0=340\cdot 10^{-6} \,{(\text{g/mm}^3)},\notag
\end{align}
and we take the wave number $k=0.5 \,(\text{mm}^{-1}).$

\subsection{Numerical implementation for Love waves}\label{NumLove}

Imposing that the roots  of the characteristic equation \eqref{x9} not to be real, one can determine the speed $v$ by numerical calculation. We start form $0$ with a small step  {size} $10^{-12}$ and we find the first value of $v$ for which \eqref{x9} admits a real solution. This is the numerical approximation of the limiting speed. For the considered material the value of the limiting speed is approximated to be   $\widehat{v}\in (6.77866,6.77867)\,\text{m/s}$.   

To find the solution of the secular equation \eqref{secularEq} we compute $  \boldsymbol{\mathcal{M}}_v$ using \eqref{fmv} by interpolating over points in  $[0, \widehat{v})$. Since we have remarked that the solutions  {are} in a small  {neighbourhood} of the limiting speed,  {we have considered for interpolation of 100 intermediate points} which are in a very small neighbourhood of the limiting speed. In this way we find the approximate solution of the secular equation $v_0=1.52374204511`\,\text{m/s}.$

\begin{figure}[h!]
	\centering
	\includegraphics[width=7cm]{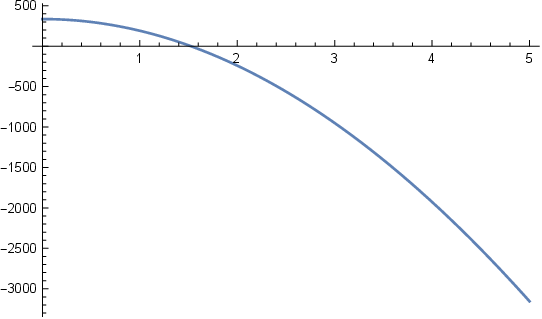}
	\put(4,94){\footnotesize $v$}
	\put(-215,125){\footnotesize $\det \boldsymbol{\mathcal{M}}_v$}
	\caption{\footnotesize A plot  $(v,\det  {  \boldsymbol{\mathcal{M}}}_v)$ obtained by evaluation for the given equidistant values for speed for dense polyurethane Foam  @0.18mm. The function $v\mapsto\det\boldsymbol{\mathcal{M}}_v$ is  a decreasing function of the wave speed $v$.}
	\label{fig:method}
\end{figure} 

Using the main theorem \ref{fmth}, we find the corresponding matrix $ \boldsymbol{\mathcal{M}}_{v_0}$ to be 
\begin{align}
		\boldsymbol{\mathcal{M}}_{v_0}=\left(
		\begin{array}{ccc}
			78.92996132702877 & 1.571343539151161 i & 1.971330465220131 \\
			-1.571343539151161 i & 8.968058122320638 & -9.11569303802329 i \\
			1.971330465220131 & 9.1156930380232 i & 9.267537853294720 \\
		\end{array}
		\right).
		\end{align}
 The matrix $\boldsymbol{\mathcal{E}}_{v_0}=\boldsymbol{\mathcal{X}}^{-1}(\boldsymbol{\mathcal{M}}_{v_0}+ {\rm i}\, \,\boldsymbol{\mathcal{Y}}^T)$ is then numerically approximated by
\begin{align}
		\boldsymbol{\mathcal{E}}_{v_0}=\left(
		\begin{array}{ccc}
			78.92996132702877 & 1.571343539151161 i & 1.971330465220131 \\
			-1.571343539151161 i & 8.968058122320638 & -9.11569303802329 i \\
			1.9713304652201319 & 9.11569303802329 i & 9.267537853294720 \\
		\end{array}
		\right).
		\end{align}
		\begin{figure}[h!]
		\centering 
		\begin{subfigure}{.31\textwidth}
			\includegraphics[width=\linewidth]{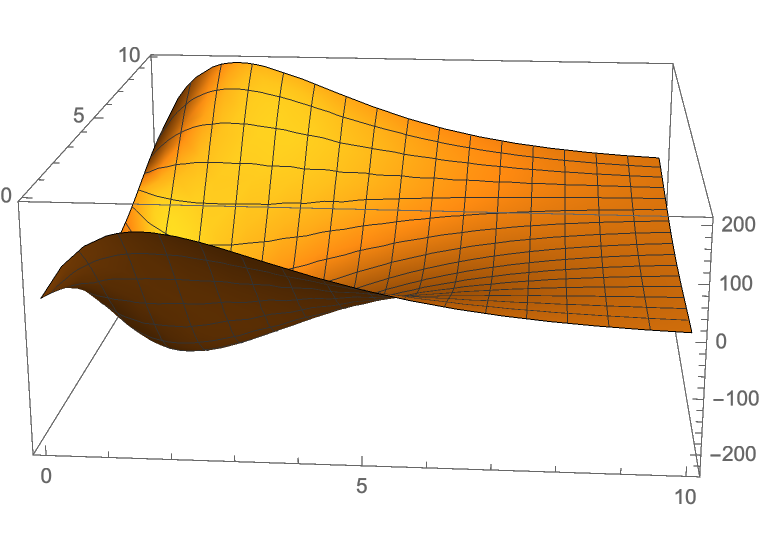}
			\put(-80,5){\footnotesize $x_2$}
			\put(-8,80){\footnotesize $x_1$}
			\put(-163,72){\footnotesize $u_3$}
			\caption{Plot of the $u_3$-component of the displacement.}
		\end{subfigure}\qquad\qquad\qquad 
		\begin{subfigure}{.31\textwidth}
			\includegraphics[width=\linewidth]{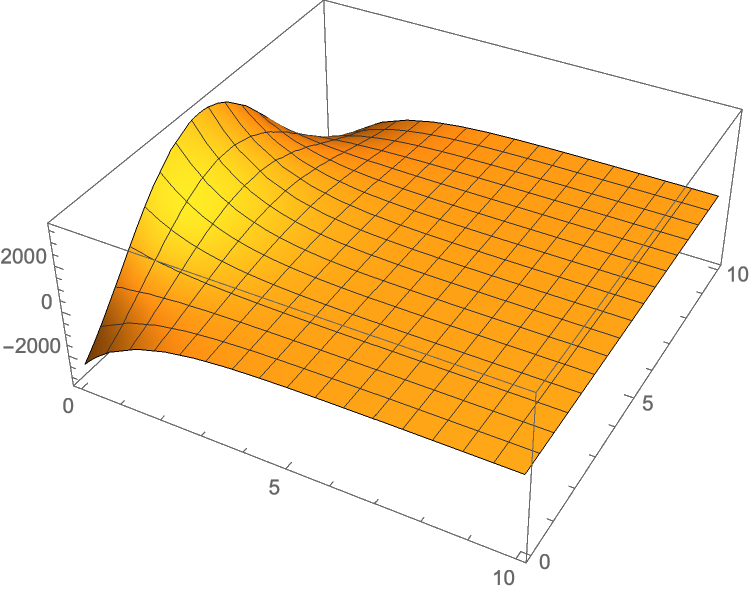}
				\put(-100,13){\footnotesize $x_2$}
			\put(-18,30){\footnotesize $x_1$}
			\put(-152,75){\footnotesize $\vartheta_1$}
			\caption{Plot of $\vartheta_1$-component of the micro-rotation vector.}
		\end{subfigure}\qquad \
		\begin{subfigure}{.31\textwidth}
			\includegraphics[width=\linewidth]{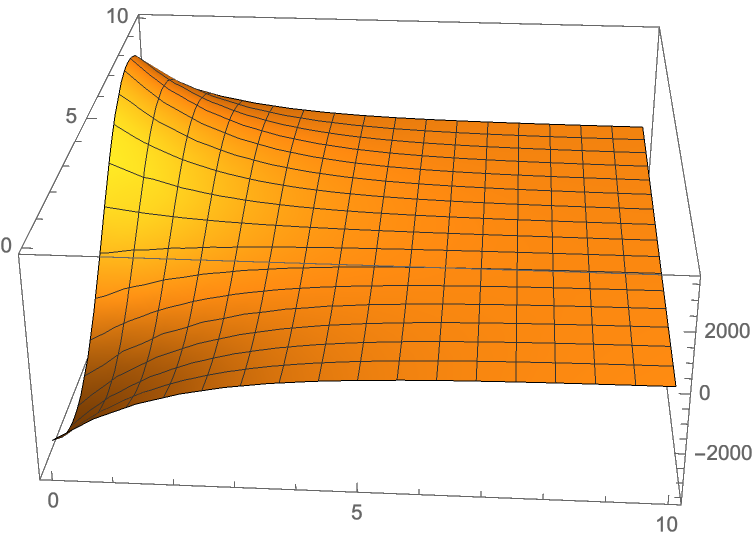}
				\put(-80,-2){\footnotesize $x_2$}
			\put(-9,80){\footnotesize $x_1$}
			\put(-163,60){\footnotesize $\vartheta_2$}
			\caption{Plot of $\vartheta_2$-component of the micro-rotation vector.}
		\end{subfigure}
		\caption{The plot of the solution at time $t=1$ and for the choice $	y_1=1\,$.}\label{u3theta1theta22}
\end{figure}
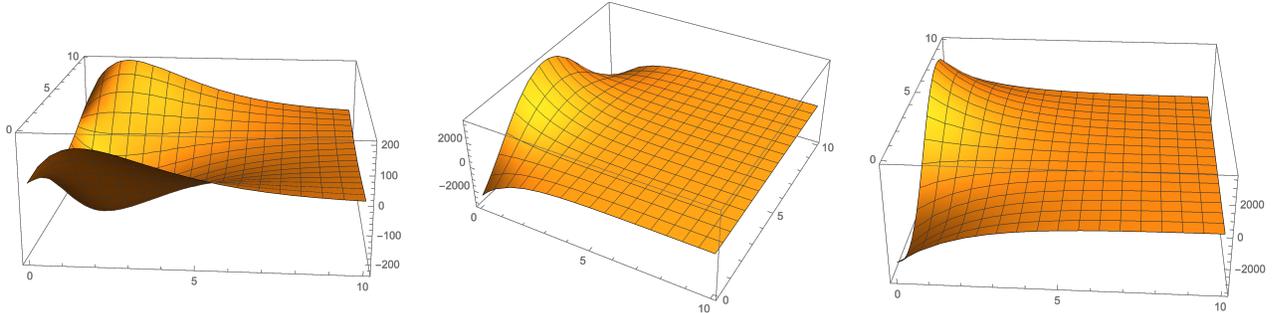

We determine $y(x_2)$ in the form given by (\ref{08}), i.e., a solution of the initial value problem (\ref{09}), after we find its approximate initial value
\begin{align}
			y(0)=\left(
			\begin{array}{c}
				y_1 \\
				-213.52544299420 i y_1 \\
				-210.23962949514 y_1 \\
			\end{array}
			\right), \qquad 	y_1\in \mathbb{C}.
\end{align}

After a substitution $y(x_2):=\widehat{\id}^{1/2}\,z(x_2)$ which lead us to $z(x_2)$ and then to the solution describing the Love waves propagation
\begin{align}\notag
 \mathcal{U}(x_1,x_2,t)=\begin{footnotesize}\begin{pmatrix}u_3(x_1,x_2,t)
 	\\
 	\vartheta_1(x_1,x_2,t)
 	\\\vartheta _2(x_1,x_2,t)
 	\end{pmatrix}\end{footnotesize}={\rm Re}\left[\begin{footnotesize}\begin{pmatrix} {\rm i}\,z_3(x_2)
 	\\
 	z_1(x_2)
 	\\z_2(x_2)
 \end{pmatrix}\end{footnotesize} e^{ {\rm i}\, k\, (  x_1-vt)}\right],
 \end{align}
of the problem given by \eqref{x6} and \eqref{x17}, see Figure \ref{u3theta1theta22}.

\subsection{Numerical implementation for Rayleigh waves}\label{NumRayleigh}

For comparison, we approximate the solution of the Rayleigh waves. The similar approach for the case of Rayleigh waves in isotropic elastic Cosserat materials can be found in \cite{khan2022existence}.

We find the the limiting speed is such that $\widehat{v}\in [6.78866,6.78868]\,\text{m/s}$, while  the approximate solution of the secular equation is 
$v_0=6.785867864122372\,\text{m/s}.$

We give the plot of  the approximate solution $\mathcal{U}=(u_1,u_2,\vartheta_3)$ for Rayleigh waves in Figure \ref{u1u2vartheta3}.

\begin{figure}[h!]
		\centering 
		\begin{subfigure}{.31\textwidth}
			\includegraphics[width=\linewidth]{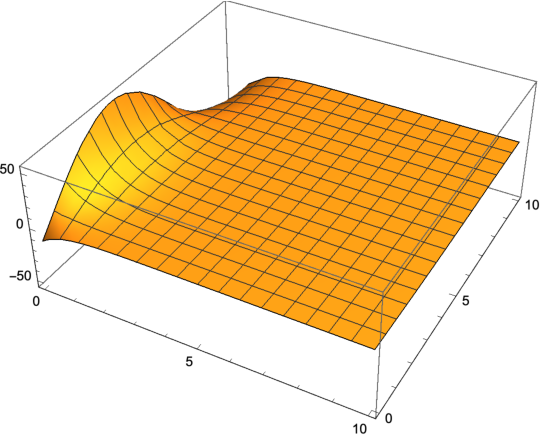}
				\put(-110,15){\footnotesize $x_2$}
			\put(-18,30){\footnotesize $x_1$}
			\put(-163,72){\footnotesize $u_1$}
			\caption{Plot of the $u_1$-component of the displacement.}
		\end{subfigure}\qquad\qquad\qquad 
		\begin{subfigure}{.31\textwidth}
			\includegraphics[width=\linewidth]{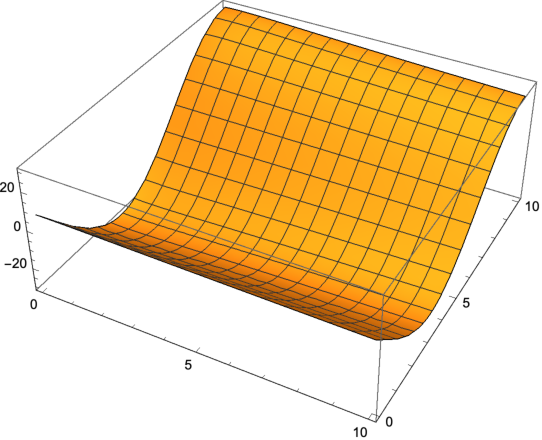}
				\put(-110,15){\footnotesize $x_2$}
			\put(-18,30){\footnotesize $x_1$}
			\put(-163,72){\footnotesize $u_2$}
			\caption{Plot of $u_2$-component of the micro-rotation vector.}
		\end{subfigure}\quad \
		\begin{subfigure}{.31\textwidth}
			\includegraphics[width=\linewidth]{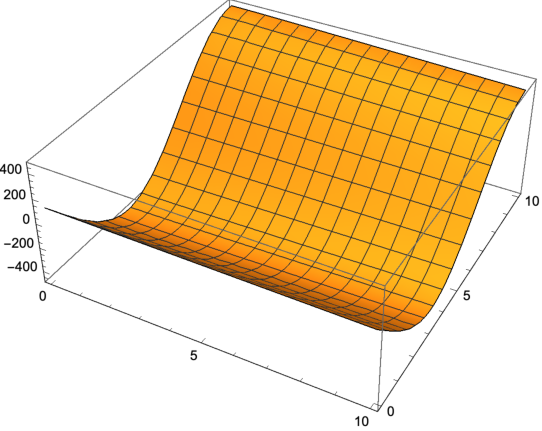}
				\put(-110,15){\footnotesize $x_2$}
			\put(-18,30){\footnotesize $x_1$}
			\put(-163,72){\footnotesize $\vartheta_3$}
			\caption{Plot of $\vartheta_3$-component of the micro-rotation vector.}
		\end{subfigure}
		\caption{The plot of the solution $\mathcal{U}=(u_1,u_2,\vartheta_3)$ in the case of Rayleigh waves at time $t=1$ and for the choice $	\varsigma={\rm i}\,$.}\label{u1u2vartheta3}
\end{figure}

In the end of this subsection we propose different values of $J$  and we calculate the speed of wave that gave the solution for the secular equation (\ref{x16}) for Love waves and also for Rayleigh wave, see Tables \ref{tab1} and \ref{tab2}.

\begin{table}[h!]
	\begin{center}
		\begin{tabular}{|c|c|c|c|}
			\hline
			    & J=0.1               & J=1              & J=10           \\ \hline
			speed of Love waves & 14.0123177251145498 & 4.79994344526609 & 1.523742045114 \\ \hline
		speed of Rayleigh waves & 17.591996478068255 & 17.31368477726568 & 6.78586786412209395 \\ \hline
		\end{tabular}
		\caption{The values of solution of the secular equation for different $J$ when $k=0.5$ in the case of Love and Rayleigh waves.}\label{tab1}
	\end{center}
\end{table}

\begin{table}[h!]
\begin{center}
\begin{tabular}{|c|c|c|c|}
\hline
  & k=0.5            & k=5                \\ \hline
speed of Love waves  & 4.79994344526609 & 16.352013950835141    \\ \hline
	speed of Rayleigh waves & 17.31368477726568 & 16.37076505452122395   \\ \hline
\end{tabular}
\caption{The values of solution of the secular equation for different $k$ when $J=1$ for Love and Rayleigh waves. }
\label{tab2}\end{center}
\end{table}

\section{From linear Cosserat theory to classical linear elasticity: a consistency check and comparison of the results}\label{Classic}\setcounter{equation}{0}

In this section we rediscover the results from classical linear elasticity as a limit case of the results obtained in the linear Cosserat theory. First, we observe that if $\mu_c\to 0$, then 
the total energy \eqref{energy} is decoupled into the energy due to the macro-displacement $u$
 \begin{align}
	W^{\rm macro}(u ):=&\,\rho\,\|u_{,t}\|^2+ \mu_{\rm e} \,\lVert \dev_3\, \sym\,\mathrm{D}u\rVert ^{2}+\frac{2\,\mu_{\rm e} +3\,\lambda_{\rm e} }{6}\left[\mathrm{tr} \left(\mathrm{D}u\right)\right]^{2},
\end{align}
and 
the energy due to the micro-rotation $\vartheta$ 
\begin{align}&
		W^{\rm micro}(\vartheta ):=\,\rho\,j\,\mu_{\rm e}\,\tau_{\rm c}^2\,\|\vartheta_{,t}\|^2+\frac{\mu_{\rm e}L_{\rm c}^2}{2}\left[{a_1 } \|\dev\,\sym \,\mathrm{D} \vartheta)\|^2+{a_2}\|\skw \,\mathrm{D}\vartheta\|^2+ \frac{4\,a_3}{3}\,[\tr(\mathrm{D} \vartheta)]^2\right].
\end{align}

Then, we have that the total energy   remains finite for
$
	L_{\rm c}\to \infty,
$
 if 
$ \|\mathrm{D}\vartheta\|^2=0$ 
		 which further implies that $ 
\vartheta=\textrm{constant}$  \textrm{in} $\Omega.$
This situation corresponds to a time dependent rigid (macroscopic) movement of the entire body. In addition, under Dirichlet homogeneous boundary conditions on $ \vartheta$ we find that $ \vartheta$ vanishes in the entire body at any time.

Moreover, assuming also that 
$ \tau_{\rm c}\to 0,$ too, the inertia term due to the micro-rotation vanishes. 

Therefore, all these limit cases together  lead us   back in the framework of classical linear elasticity, i.e., the  elastic energy density is
\begin{align}
	W_{\rm classic}(\mathrm{D}u )=&\,  \mu_{\rm e} \,\lVert \dev_3\, \sym\,\mathrm{D}u \rVert ^{2}+\frac{2\,\mu_{\rm e} +3\,\lambda_{\rm e} }{6}\left[\mathrm{tr} \left(\mathrm{D}u \right)\right]^{2},\qquad \qquad \qquad \qquad \qquad 
\end{align}
while the equations of motion are the Lam\'e equations
\begin{align}\label{PDE1}
	\rho\,\frac{\partial^2 \,u_i}{\partial\, t^2}&=\underbrace{\mu_{\rm e} \dd\sum_{l=1}^3\frac{\partial^2 u_i }{\partial x_l^2}+(\mu_{\rm e} +\lambda_{\rm e} )\dd\sum_{l=1}^3\frac{\partial^2 u_l }{\partial x_l\partial x_i}}_{={\rm Div}\boldsymbol{\sigma}_{\rm classic}},
\end{align}
and the decoupled equations of micro-rotations
\begin{align}\label{PDE2}
		\rho\, j\,\mu_{\rm e}\,\tau_{\rm c}^2\,\frac{\partial^2 \,\vartheta_i}{\partial\, t^2}&=\underbrace{{\mu_{\rm e}L_{\rm c}^2}\,\left[(\alpha_1+\alpha_2)\dd\sum_{l=1}^3\frac{\partial^2 \vartheta _i }{\partial x_l^2}+(\alpha_1-\alpha_2+\alpha_3)\dd\sum_{l=1}^3\frac{\partial^2 \vartheta_l }{\partial x_l\partial x_i}\right]}_{={\rm Div}\mathbf{ m}}, \notag
	\end{align}
	with $i=1,2,3$ and
\begin{align}
	\boldsymbol{\sigma}_{\rm classic}&:=2\,\mu_{\rm e}\,\sym \,{\rm D}u \, +\lambda_{\rm e} \,\tr (\sym \,{\rm D}u)\, \id,\\
		\boldsymbol{m}_{\rm microrotation}&:={\mu_{\rm e}\,L_{\rm c}^2}\,\left[2\,{\alpha_1}\, \sym\,{\rm D}\vartheta +2\,{\alpha_2}\,\skw\, {\rm D}\vartheta+\alpha_3\,\tr({\rm D}\vartheta)\,\id\right].\notag
\end{align}

Letting $\mu_{\rm c}\to 0$, the matrices involved in the Riccati equation become
\begin{align}\label{Xc}
	\boldsymbol{\mathcal{X}}:=&\begin{footnotesize}\left(
		\begin{array}{ccc}
			\frac{k^2 \mu _e}{\rho } & 0 & 0 \\
			0 & \frac{k^2L_c^2 \left(\alpha _1  +\alpha _2  \right)}{\rho  \tau _c^2 j} & 0 \\
			0 & 0 & \frac{k^2L_c^2 \left(2 \alpha _1 +\alpha _3  \right)}{\rho  \tau _c^2 j} \\
		\end{array}
		\right)\end{footnotesize}, \qquad \qquad 
	\boldsymbol{\mathcal{Y}}:= \begin{footnotesize}\left(
		\begin{array}{ccc}
			0 & 0 & 0 \\
			0 & 0 & \frac{\alpha _3 L_c^2 k  }{\rho  \tau _c^2 j} \\
			0 & \frac{kL_c^2 \left(\alpha _1  -\alpha _2 \right)}{\rho  \tau _c^2 j} & 0 \\
		\end{array}
		\right)\end{footnotesize},\notag\\
	\widetilde{	\boldsymbol{\mathcal{Z}}}:=&\begin{footnotesize}\left(
		\begin{array}{ccc}
			\frac{k^2 \mu _e}{\rho }-k^2 v^2 & 0 & 0 \\
			0 & \frac{k^2 L_c^2\left(2 \alpha _1  +\alpha _3 \right)}{\rho  \tau _c^2 j}-k^2 v^2 & 0 \\
			0 & 0 & \frac{k^2L_c^2 \left(\alpha _1  +\alpha _2 \right)}{\rho  \tau _c^2 j}-k^2 v^2 \\
		\end{array}
		\right)\end{footnotesize}.
\end{align}

Thus, we have
\begin{align}
	\boldsymbol{\mathcal{X}}_{\!\! \theta}&=
	\begin{footnotesize}\left(
		\begin{array}{ccc}
			\frac{k^2 \left(\mu _e-\rho  v^2 \sin ^2(\theta )\right)}{\rho } & 0^T \\
			0 &  \boldsymbol{\mathcal{A}}_\theta 
		\end{array}
		\right)\end{footnotesize},\qquad \qquad
	\boldsymbol{\mathcal{Y}}_\theta=
	\begin{footnotesize}\left(
		\begin{array}{ccc}
			k^2 v^2 \sin (\theta ) \cos (\theta ) & 0^T \\
			0 & \boldsymbol{\mathcal{B}}_\theta 
		\end{array}
		\right)\end{footnotesize},
\end{align}
where
\begin{align}
	\boldsymbol{\mathcal{A}}_\theta&=
	\begin{tiny}\left(
		\begin{array}{ccc}
			k^2 \left(\alpha _2 c^2 \cos ^2(\theta )-\frac{1}{2} \alpha _1 c^2 (\cos (2 \theta )-3)+\sin ^2(\theta ) \left(\alpha _3 c^2-v^2\right)\right) & \left(\alpha _1-\alpha _2+\alpha _3\right) \left(-c^2\right) k \sin (\theta ) \cos (\theta ) \\
			\left(\alpha _1-\alpha _2+\alpha _3\right) \left(-c^2\right) k \sin (\theta ) \cos (\theta ) & \frac{1}{2} k^2 \left(c^2 \left(\alpha _1 (\cos (2 \theta )+3)+2 \left(\alpha _2 \sin ^2(\theta )+\alpha _3 \cos ^2(\theta )\right)\right)-2 v^2 \sin ^2(\theta )\right) \\
		\end{array}
		\right)\end{tiny}\\
	\boldsymbol{\mathcal{B}}_\theta&=
	\begin{tiny}\left(
		\begin{array}{ccc}
			k^2 \sin (\theta ) \cos (\theta ) \left(v^2-\left(\alpha _1-\alpha _2+\alpha _3\right) c^2\right) & c^2 k \left(\left(\alpha _2-\alpha _1\right) \sin ^2(\theta )+\alpha _3 \cos ^2(\theta )\right) \\
			c^2 k \left(\left(\alpha _1-\alpha _2\right) \cos ^2(\theta )-\alpha _3 \sin ^2(\theta )\right) & k^2 \sin (\theta ) \cos (\theta ) \left(\left(\alpha _1-\alpha _2+\alpha _3\right) c^2+v^2\right) \\
		\end{array}
		\right)\end{tiny},\notag
\end{align}
with the notation $c=\frac{ L_c^2}{\rho  \tau _c^2 j}$.

If we impose the boundary condition
\begin{align}\label{03}
	\boldsymbol{\sigma}_{\rm classic}.\, n=0, \qquad \qquad\textrm{for}\quad x_2=0,
\end{align}
and the following   decay condition
\begin{align}\label{04}
	\lim_{x_2 \rightarrow \infty}\{u_3,\sigma_{13}^{\rm classic},\sigma_{23}^{\rm classic}\} (x_1,x_2,t)=0 \qquad \quad \forall\, x_1\in\mathbb{R}, \quad \forall \, t\in[0,\infty),
\end{align}
by  following the usual path from classical elasticity, i.e., by looking for a solution \begin{align}\label{cx5}
	\mathcal{U}_{\rm classic}(x_1,x_2,t)=\begin{footnotesize}\begin{pmatrix}u_1(x_1,x_2,t)
			\\
			u_2(x_1,x_2,t)
			\\u_3(x_1,x_2,t)
	\end{pmatrix}\end{footnotesize}=\begin{footnotesize}\begin{pmatrix}0
	\\
	0
	\\d_3^{\rm classic}(x_1,x_2,t)
\end{pmatrix}\end{footnotesize}={\rm Re}\left[\begin{footnotesize}\begin{pmatrix} 0
			\\
		0
			\\ d_3^{\rm classic}
	\end{pmatrix}\end{footnotesize} \,e^{{\rm i}\,r\,k\,x_2}\,e^{ {\rm i}\, k\, (  x_1-vt)}\right], \quad \text{Im}\,r>0,
\end{align}
of \eqref{PDE1}, 
it is easy to see \cite{Achenbach} that  
\begin{align}
	r=\sqrt{1-\left(\frac{v}{c_T}\right)^2}, \qquad \text{where}\quad c_T=\sqrt{\frac{\mu_{\rm e}}{\rho}}.
\end{align}
Moreover, since from the free surface boundary  conditions on $x_2=0$ we have
\begin{align}
	\frac{\partial u_3}{\partial x_2}\Big|_{x_2=0}=0,
\end{align}
we find that either $d_3^{\rm classic}=0$ or $r=0$, neither case representing a surface wave, see \cite[page 219]{Achenbach}. In conclusion, using the classical linear theory of isotropic  elastic materials it is not possible to analytically find  that surface waves with displacements perpendicular to the plane of propagation araise \cite{Achenbach}. However, experimental data shows that these kind of waves may occur along a free surface. A compromise has been introduced in classical elasticity, see \cite[page 219]{Achenbach}, by considering that the half space is covered by a thin layer of a different material. For this new problem it is proved that Love-type waves orise \cite[page 219]{Achenbach}. However, there is no need to consider layers in order to explain the existence of Love waves from experiments, since the presence of the Love waves could be explained theoretical by the presence of micro-rotations, in the sense that even in the uncoupled case ($\mu_{\rm c}=0$) the displacement $u_3$ may not vanish for a speed of propagation produced by the microstructure. We will explain this in the rest of this section.

We may argue that $c_T$ is not an admissible speed of propagation in the sense of our approach, i.e., by showing that it is bigger than the limiting speed. Now, 	by the {limiting speed} we understand a speed $\widehat{v}>0$, such that for all wave speeds satisfying 	$0\le v<\widehat{v}$  (subsonic speeds) the roots  of the characteristic equation \eqref{x9} corresponding to the matrices given by \eqref{Xc} are not real and vice versa, i.e., if the roots  of the characteristic equation \eqref{Xc}  are not real then they correspond to wave speeds $v$ satisfying 	$0\le v<\widehat{v}$.

Indeed, similar arguments as those used for the proof of  Proposition \ref{proprealw2} and Proposition \ref{lemmaGH} lead us to
\begin{proposition}\label{proprealw2c} Let   ${\xi}\in \mathbb{R}^3$, $\xi\neq 0$ be any direction of the form $\xi=(\xi_1,\xi_2, 0)^T$. The necessary and sufficient conditions for existence of   a non trivial solution of the system of partial differential equations \eqref{PDE1}, \eqref{PDE2}  of the form given by \eqref{rwr} are 
	\begin{align}\label{d12f}
		2\alpha_1+\alpha_3 >0,\qquad \qquad \mu_{\rm e} >0,\qquad \qquad   \alpha_1+\alpha_2>0.
	\end{align}
\end{proposition}
Let  {us} assume that  the constitutive coefficients satisfy the conditions 	\eqref{d12}. Then, there	exists a limiting speed $\widehat{v}>0$. Furthermore,   if one root   $r_v$ of the characteristic equation \eqref{x9} coresponding to the matrices given by \eqref{Xc} is real then it corresponds to a speed $v\geq \widehat{v}$ (non-admissible).
	Moreover,  for all $\theta\in \left(-\frac{\pi}{2},\frac{\pi}{2}\right)$, $k>0$ and $0\le v<\widehat{v}$, the tensor $\widetilde{\boldsymbol{\mathcal{Z}}}_\theta:=\sin^2\theta\,\boldsymbol{\mathcal{X}}+\sin\theta\cos\theta (\boldsymbol{\mathcal{Y}}+\boldsymbol{\mathcal{Y}}^T)+\cos^2\theta\,\boldsymbol{\mathcal{Z}}-k^2 v^2 \cos^2\theta \,{\id}$ is positive definite. Since $ \boldsymbol{\mathcal{X}}_{\!\! \theta}(\theta+\frac{\pi}{2})=\widetilde{\boldsymbol{\mathcal{Z}}}_\theta$, we find that for all $\theta\in \left(-\frac{\pi}{2},\frac{\pi}{2}\right)$, $k>0$ and $0\le v<\widehat{v}$, the tensor ${\boldsymbol{\mathcal{X}}}_\theta$ must be positive definite.   Therefore, we must have
	\begin{align}
		\frac{k^2 \left(\mu _e-\rho  v^2 \sin ^2(\theta )\right)}{\rho }>0\qquad \forall \ \theta\in \left(-\frac{\pi}{2},\frac{\pi}{2}\right), \quad \forall\  k>0,\quad \forall \ 0\le v<\widehat{v},
	\end{align}
	which implies that
	\begin{align}
		v<c_T=\sqrt{\frac{\mu_{\rm e}}{\rho}}\qquad \forall \ 0\le v<\widehat{v},
	\end{align}
	i.e., it implies that 	$\widehat{v}\leq c_T$ and that an admissible speed of propagation may never reach the value $c_T$.

Consequently, from \eqref{fmv} we deduce that the matrix $\boldsymbol{\mathcal{M}}$ giving the secular equation
$
	\det \boldsymbol{\mathcal{M}}=0
$
is
\begin{align}
	\boldsymbol{\mathcal{M}}&=
	\begin{footnotesize}\left(
		\begin{array}{ccc}
		 \Big[\dd\int_{0}^{\pi}	 	\Big(\frac{k^2 \left(\mu _e-\rho  v^2 \sin ^2(\theta )\right)}{\rho }\Big) ^{-1}\, d\theta\Big]^{-1}\Big[\pi\id-  {\rm i}\, \int_{0}^{\pi}    {	\Big(\frac{k^2 \left(\mu _e-\rho  v^2 \sin ^2(\theta )\right)}{\rho }\Big)} ^{-1}   k^2 v^2 \sin (\theta ) \cos (\theta )\, d\theta\Big]& 0^T \\
			0 &  \boldsymbol{\mathcal{N}} 
		\end{array}
		\right)\end{footnotesize},
\end{align}
where
	\begin{align}\label{fmv0}
	\boldsymbol{\mathcal{N}}=\Big(\int_{0}^{\pi}	  \boldsymbol {\mathcal{A}}_  { \theta} ^{-1}\, d\theta\Big)^{-1}\Big(\pi\id-  {\rm i}\, \int_{0}^{\pi}   \boldsymbol{\mathcal{A}}_  { \theta} ^{-1}   {   \boldsymbol{\mathcal{B}}}_\theta  ^T\, d\theta\Big),
\end{align}

We have
\begin{align}\label{notz}
\Big[\dd\int_{0}^{\pi}	 	\Big(\frac{k^2 \left(\mu _e-\rho  v^2 \sin ^2(\theta )\right)}{\rho }&\Big) ^{-1}\, d\theta\Big]^{-1}\Big[\pi\id-  {\rm i}\, \int_{0}^{\pi}    {	\Big(\frac{k^2 \left(\mu _e-\rho  v^2 \sin ^2(\theta )\right)}{\rho }\Big)} ^{-1}   k^2 v^2 \sin (\theta ) \cos (\theta )\, d\theta\Big]\\ &=	\frac{2 i k^2 \sqrt{\mu _e} \sqrt{\rho  v^2-\mu _e}}{\rho  \left(2 \mu _e-\rho  v^2+2 \pi \right)}\neq 0.\notag
\end{align}
Therefore, since we already know that for $\mu_{\rm c}\to 0$ (from the same arguments as for the case $\mu_{\rm c}>0$) there is a unique solution of the secular equation $\det \boldsymbol{\mathcal{N}}$, using \eqref{notz} we have that the secular equation reduces to 
\begin{align}
	\det \boldsymbol{\mathcal{N}}=0,
	\end{align}
	where  $\boldsymbol{\mathcal{N}}$ is given by \eqref{fmv0} and that there is a unique solution of this equation which satisfies $0\leq v<\widehat{v}$. In particular the solution $v$ is not $c_T$, which implies that $u_3$ does not vanish, i.e., that we have a theoretical meaning of the existence of the Love waves in experiments, without considering additional material layers as it is done before \cite{Achenbach}.  {This confirms and proofs rigorously  the remark made in  \cite{kulesh2009problem,Kulesh1,Kulesh2} about the existence of a qualitatively new wave mode   for the  Cosserat elastic model.}  The meaning is that the non-zero macro-displacement $u_3$ is produced by a speed of propagation which is due to the given microstructure.
		
The exact form of $\boldsymbol{\mathcal{N}}$  is
\begin{align}
\!\!\begin{tiny}\boldsymbol{\mathcal{N}}\!=\!\!	\left(\!\!\!
	\begin{array}{cc}
		\frac{c v^2}{c^2 \sqrt{\frac{2 \alpha _1+\alpha _3}{\left(2 \alpha _1+\alpha _3\right) c^2-v^2}}-\sqrt{c^2-\frac{v^2}{\alpha _1+\alpha _2}}} & -\frac{i c^2 \sqrt{\frac{2 \alpha _1+\alpha _3}{\left(2 \alpha _1+\alpha _3\right) c^2-v^2}} 2 \alpha _1\left( c^2- \sqrt{\frac{\left(c^2-\frac{v^2}{\alpha _1+\alpha _2}\right) \left(\left(2 \alpha _1+\alpha _3\right) c^2-v^2\right)}{2 \alpha _1+\alpha _3}}-v^2\right)}{c^2 \sqrt{\frac{2 \alpha _1+\alpha _3}{\left(2 \alpha _1+\alpha _3\right) c^2-v^2}}-\sqrt{c^2-\frac{v^2}{\alpha _1+\alpha _2}}} \vspace{2mm}\\
		\frac{i c^2 \left(-v^2 \sqrt{\frac{\alpha _1+\alpha _2}{\left(\alpha _1+\alpha _2\right) c^2-v^2}}-2 \alpha _1 \left(\sqrt{c^2-\frac{v^2}{2 \alpha _1+\alpha _3}}-c^2 \sqrt{\frac{\alpha _1+\alpha _2}{\left(\alpha _1+\alpha _2\right) c^2-v^2}}\right)\right)}{c^2 \sqrt{\frac{\alpha _1+\alpha _2}{\left(\alpha _1+\alpha _2\right) c^2-v^2}}-\sqrt{c^2-\frac{v^2}{2 \alpha _1+\alpha _3}}} & \frac{c v^2}{c^2 \sqrt{\frac{\alpha _1+\alpha _2}{\left(\alpha _1+\alpha _2\right) c^2-v^2}}-\sqrt{c^2-\frac{v^2}{2 \alpha _1+\alpha _3}}} \\
	\end{array}
	\!\!\!\right)\!\!.\end{tiny}.
\end{align}
Let us remark that since for all $v\in[0,\widehat{v})$ the matrix $ \boldsymbol {\mathcal{A}}_  { \theta}$ is also positive definite $\forall \theta\in [0,\pi]$,  then $\int_{0}^{\pi}	  \boldsymbol {\mathcal{A}}_  { \theta} ^{-1}\, d\theta$ is positive definite. Since 
\begin{align}
\int_{0}^{\pi}	  \boldsymbol {\mathcal{A}}_  { \theta} ^{-1}\, d\theta=\left(
\begin{array}{cc}
	\frac{\pi  \left(c^2 \sqrt{\frac{2 \alpha _1+\alpha _3}{\left(2 \alpha _1+\alpha _3\right) c^2-v^2}}-\sqrt{c^2-\frac{v^2}{\alpha _1+\alpha _2}}\right)}{c v^2} & 0 \\
	0 & \frac{\pi  c^2 \sqrt{\frac{\alpha _1+\alpha _2}{\left(\alpha _1+\alpha _2\right) c^2-v^2}}-\pi  \sqrt{c^2-\frac{v^2}{2 \alpha _1+\alpha _3}}}{c v^2} \\
\end{array}
\right)\in \mathbb{R}^{2\times 2}
	\end{align}
it follows that for all $v\in[0,\widehat{v})$ we must have
\begin{align}\label{limp}
	\left(2 \alpha _1+\alpha _3\right) c^2-v^2>0,\   \left(\alpha _1+\alpha _2\right) c^2-v^2>0 \ \text{and} \ c^2 \sqrt{\frac{2 \alpha _1+\alpha _3}{\left(2 \alpha _1+\alpha _3\right) c^2-v^2}} \sqrt{\frac{\alpha _1+\alpha _2}{\left(\alpha _1+\alpha _2\right) c^2-v^2}}>1.
\end{align}
The third inequality is redundant, while if \eqref{limp}$_{1,2}$ are satisfied, then $\boldsymbol{\mathcal{N}}$ is an hermitian matrix, while the secular equation
is given by
$
	f_c(v)=0,
$
with $f_c:[0,\widehat{v})\to \mathbb{R}$,
\begin{footnotesize}
\begin{align}
	f_c(v):=&\begin{footnotesize}\frac{1}{\left(c^2 \sqrt{\frac{\alpha _1+\alpha _2}{\left(\alpha _1+\alpha _2\right) c^2-v^2}}-\sqrt{c^2-\frac{v^2}{2 \alpha _1+\alpha _3}}\right) \left(\sqrt{c^2-\frac{v^2}{\alpha _1+\alpha _2}}-c^2 \sqrt{\frac{2 \alpha _1+\alpha _3}{\left(2 \alpha _1+\alpha _3\right) c^2-v^2}}\right)}\end{footnotesize}\notag\\
	& \begin{tiny}\times \Bigg\{c^2 v^4 \left(c^2 \sqrt{\frac{\alpha _1+\alpha _2}{\left(\alpha _1+\alpha _2\right) c^2-v^2}} \sqrt{\frac{2 \alpha _1+\alpha _3}{\left(2 \alpha _1+\alpha _3\right) c^2-v^2}}-1\right)\end{tiny}\notag\\&\qquad \begin{tiny}+2 \alpha _1 c^4 \Big(v^2 \left(\sqrt{\frac{\alpha _1+\alpha _2}{\left(\alpha _1+\alpha _2\right) c^2-v^2}} \left(\sqrt{\frac{\left(\left(\alpha _1+\alpha _2\right) c^2-v^2\right) \left(\left(2 \alpha _1+\alpha _3\right) c^2-v^2\right)}{\left(\alpha _1+\alpha _2\right) \left(2 \alpha _1+\alpha _3\right)}}-2 c^2\right) +\sqrt{c^2-\frac{v^2}{2 \alpha _1+\alpha _3}}\right)\end{tiny}\notag\\
	&\qquad \begin{tiny}-2 \alpha _1 \left(c^2 \sqrt{\frac{\alpha _1+\alpha _2}{\left(\alpha _1+\alpha _2\right) c^2-v^2}}-\sqrt{c^2-\frac{v^2}{2 \alpha _1+\alpha _3}}\right) \left(\sqrt{\frac{\left(\left(\alpha _1+\alpha _2\right) c^2-v^2\right) \left(c^2-\frac{v^2}{2 \alpha _1+\alpha _3}\right)}{\alpha _1+\alpha _2}}-c^2\right)\Big) \end{tiny}\\&\qquad \begin{tiny}\times \sqrt{\frac{2 \alpha _1+\alpha _3}{\left(2 \alpha _1+\alpha _3\right) c^2-v^2}}\Bigg\}\end{tiny}.\notag
\end{align}
\end{footnotesize}
Let us remark that for the assumptions from the beginning of this section, i.e., $ L_{\rm c}\to \infty$ and $ \tau_{\rm c}\to 0$ it is natural to consider the case $c\to \infty$, the expectation being that is this situation we must lead to the same result as in the classical elasticity. Indeed, we already know that there is no admissible speed due to the macrostructure, but we also have  that assuming \eqref{limp}, we have  that
 $\lim_{c\to \infty} f_c(v)=\infty$, $\forall v\in[0,\widehat{v})$. This means that in  this case there does not exist a solution of the secular equation due to the microstructure. Therefore, for Love waves propagation $\mu_{\rm c}\to 0$, $ L_{\rm c}\to \infty$ and $ \tau_{\rm c}\to 0$  means the framework of classical elasticity. 

\section{Conclusion}

This paper, together with \cite{ghiba2023cosserat}, makes possible a complete comparison between the propagation of two principal types of surface waves, Love waves and Rayleigh waves, in Cosserat homogeneous isotropic materials. From experiments we know that, in general, they exhibit fundamentally distinct modes of motion. Love waves are characterised by purely horizontal, transverse displacements perpendicular to the direction of wave propagation, with no vertical component. In contrast, Rayleigh waves induce retrograde elliptical particle motion in the vertical plane, combining both vertical and longitudinal displacements along the direction of propagation. In the standard half-space model used to describe these waveforms, the boundary is assumed to be traction-free, and all non-zero displacement and stress components are required to decay asymptotically to zero with increasing depth. The method and the algorithms considered in \cite{ghiba2023cosserat} and the present paper are easy to be implemented and mathematically fully justified, due to the previous results given by Fu and Mielke \cite{fu2001nonlinear,mielke2004uniqueness}. 

It is known that the experimental data shows that Love waves exists. It is also known that in the framework of  classical linear theory of isotropic  elastic materials it is not possible to find analytically Love waves-type solutions \cite{Achenbach}. To avoid this fact,  a compromise has been introduced in classical elasticity, see \cite[page 219]{Achenbach}, by considering that the half space is covered by a layer of a different material. In this paper, we give a complete mathematical study which shows that the Cosserat theory is able to model the Love waves propagation without  any artificial introduction of surface layers.

A complete mathematical study of the propagation of Love waves in Cosserat elastic materials which includes a proof of the existence and uniqueness of an admissible solution of the secular equation was missing until our study.
The present study confirms the effectiveness of Fu and Mielke's general method in establishing the existence and uniqueness of a subsonic solution that describes the propagation of Love waves in an isotropic half-space (without additional material surface layers), based on the linear theory of isotropic elastic Cosserat materials.  The given formulas and the given algorithm admit numerical implementation and we use them for the study of Love waves propagation in  dense polyurethane  {Foam@0.18mm}  \cite{Lakes83,lakes1987foam}. We remark that once in the measurement of Love waves and Rayleigh waves propagation on the Earth surface the Love waves propagate faster than the Rayleigh waves, for our material modelled in the framework of the Cosserat linear theory we obtain the opposite. For other materials modelled by the Cosserat linear theory, numerical simulations show that there is no general rule about which of those type of seismic waves travel faster, this 
property depending on the material parameters  {and needs further research}. We have also seen that for Love waves propagation the case $\mu_{\rm c}\to 0$, $ L_{\rm c}\to \infty$ and $ \tau_{\rm c}\to 0$  means the framework of classical elasticity, while, in general, for $\mu_{\rm c}\to 0$ the presence of the Love waves at macroscopic level could be an effect of the microscopic level.
\bigskip

\begin{footnotesize}
	\noindent{\bf Acknowledgements:} The  work of I.D. Ghiba  was supported by a grant of the Romanian Academy,  within GAR2023, project code 73: {\it Research grant made with financial support from the Donors' Recurrent Fund, available to the Romanian Academy and managed through the "PATRIMONIU" Foundation GAR2023}.

	\bibliographystyle{plain} 

\addcontentsline{toc}{section}{References}

\begin{thebibliography}{10}
	
	\bibitem{Achenbach}
	J.D. Achenbach.
	\newblock {\em Wave Propagation in Elastic Solids}.
	\newblock North-Holland Publishing Company, Amsterdam, 1973.
	
	\bibitem{biryukov1985impedance}
	S.V. Biryukov.
	\newblock Impedance method in the theory of elastic surface-waves.
	\newblock {\em Soviet Physics Acoustics}, 31(5):350--354, 1985.
	
	\bibitem{ChiritaGhiba2}
	S.~Chiri{\c t}{\u a} and I.~D. Ghiba.
	\newblock Inhomogeneous plane waves in elastic materials with voids.
	\newblock {\em Wave Motion}, 47:333--342, 2010.
	
	\bibitem{ChiritaGhiba3}
	S.~Chiri{\c t}{\v a} and I.~D. Ghiba.
	\newblock Rayleigh waves in {Cosserat} elastic materials.
	\newblock {\em International Journal of Engineering Science}, 51:117--127, 2012.
	
	\bibitem{Cosserat09}
	E.~Cosserat and F.~Cosserat.
	\newblock {\em Th\'eorie des corps d\'eformables.}
	\newblock Librairie Scientifique A. Hermann et Fils (engl. translation by D.
	Delphenich 2007, pdf available at
	{http://www.uni-due.de/$\sim$hm0014/Cosserat\_files/Cosserat09\_eng.pdf}),
	reprint 2009 by Hermann Librairie Scientifique, ISBN 978 27056 6920 1, Paris,
	1909.
	
	\bibitem{destrade2007seismic}
	M.~Destrade.
	\newblock Seismic {Rayleigh} waves on an exponentially graded,~orthotropic
	half-space.
	\newblock {\em Proceedings of the Royal Society A: Mathematical, Physical},
	463(2078):495--502, 2007.
	
	\bibitem{Eringen99}
	A.C. Eringen.
	\newblock {\em Microcontinuum {F}ield {T}heories.}
	\newblock Springer, Heidelberg, 1999.
	
	\bibitem{Eringen64}
	A.C. Eringen and E.S. Suhubi.
	\newblock Nonlinear theory of simple micro-elastic solids. {I}.
	\newblock {\em International Journal of Engineering Science}, 2:189--203, 1964.
	
	\bibitem{Erofeyev}
	V.~I. Erofeyev.
	\newblock {\em Wave {P}rocesses in {S}olids with {M}icrostructure}.
	\newblock World Scientific, New Jersey, 2003.
	
	\bibitem{fu2002new}
	Y.B. Fu and A.~Mielke.
	\newblock A new identity for the surface--impedance matrix and its application
	to the determination of surface-wave speeds.
	\newblock {\em Proceedings of the Royal Society A: Mathematical, Physical and Engineering Sciences},
	458(2026):2523--2543, 2002.
	
	\bibitem{fu2001nonlinear}
	Y.B. Fu and R.W. Ogden.
	\newblock {\em Nonlinear {E}lasticity: {T}heory and {A}pplications}, volume
	281.
	\newblock Cambridge University Press, 2001.
	
	\bibitem{ghiba2023cosserat}
	A.~Madeo, I.D.~Ghiba, G.~Rizzi and P.~Neff.
	\newblock Cosserat micropolar elasticity: classical  {Eringen} vs. dislocation
	form.
	\newblock {\em Journal of Mechanics of Materials and Structures},
	18(1):93--123, 2023.
	
	\bibitem{ingebrigtsen1969elastic}
	K.A. Ingebrigtsen and A.~Tonning.
	\newblock Elastic surface waves in crystals.
	\newblock {\em Physical Review}, 184(3):942, 1969.
	
	\bibitem{khan2022existence}
	H.~Khan, I.D. Ghiba, A.~Madeo, and P.~Neff.
	\newblock Existence and uniqueness of {Rayleigh} waves in isotropic elastic
	{Cosserat} materials and algorithmic aspects.
	\newblock {\em Wave Motion}, 110:102898, 2022.
	
	\bibitem{knobloch2012topics}
	H.W. Knobloch, A.~Isidori, and D.~Flockerzi.
	\newblock {\em Topics in {C}ontrol {T}heory}, volume~22.
	\newblock Birkh{\"a}user, 2012.
	
	\bibitem{Koebke}
	R.H. Koebke and Y.~Weitsman.
	\newblock Surface wave propagation over an elastic {Cosserat} half-space.
	\newblock {\em Journal of the Acoustical Society of America}, 50:875--884, 1971.
	
	\bibitem{kulesh2009problem}
	M.A. Kulesh, E.F. Grekova, and I.N. Shardakov.
	\newblock The problem of surface wave propagation in a reduced Cosserat
	medium.
	\newblock {\em Acoustical Physics}, 55(2):218--226, 2009.
	
	\bibitem{Kulesh01}
	M.A. Kulesh, V.P. Matveenko, and I.N. Shardakov.
	\newblock Exact analytical solution of the {K}irsch problem within the
	framework of the {C}osserat continuum and pseudocontinuum.
	\newblock {\em Journal of Applied Mechanics and Technical Physics}, 42:687--695, 2001.
	
	\bibitem{Kulesh03}
	M.A. Kulesh, V.P. Matveenko, and I.N. Shardakov.
	\newblock Parametric analysis of analytical solutions to one- and
	two-dimensional problems in couple-stress theory of elasticity.
	\newblock {\em Zeitschrift f\"ur Angewandte Mathematik und Mechanik},
	83:238--248, 2003.
	
	\bibitem{Kulesh05}
	M.A. Kulesh, V.P. Matveenko, and I.N. Shardakov.
	\newblock Construction and analysis of an analytical solution for the surface
	{R}ayleigh wave within the framework of the {C}osserat continuum.
	\newblock {\em Journal of Applied Mechanics and Technical Physics}, 46:556--563, 2005.
	
	\bibitem{Kulesh1}
	M.A. Kulesh, V.P. Matveenko, and I.N. Shardakov.
	\newblock On the propagation of elastic surface waves in the  {Cosserat} medium.
	\newblock {\em Doklady Physics}, 50:601--604, 2005.
	
	\bibitem{Kulesh2}
	Matveenko V.P. Shardakov~I.N. Kulesh, M.A.
	\newblock Propagation of surface elastic waves in the {Cosserat} medium.
	\newblock {\em Acoustical Physics}, 52:186--193, 2006.
	
	\bibitem{lakes1987foam}
	R.~Lakes.
	\newblock Foam structures with a negative {Poisson}'s ratio.
	\newblock {\em Science}, 235(4792):1038--1040, 1987.
	
	\bibitem{Lakes83}
	R.S. Lakes.
	\newblock Size effects and micromechanics of a porous solid.
	\newblock {\em Journal of Materials Science}, 18:2572--2580, 1983.
	
	
	\bibitem{Love1911} A.E.H.	Love, \newblock Some Problems of Geodynamics, \newblock Cambridge University Press, 1911. \textit{Being an essay to which the Adams prize in the University of Cambridge was adjudged in 1911} \texttt{{https://ia801300.us.archive.org/22/items/cu31924060184367/cu31924060184367.pdf}}
	
	\bibitem{MadeoNeffGhibaWZAMM}
	A.~Madeo, P.~Neff, I.~D. Ghiba, L.~Placidi, and G.~Rosi.
	\newblock Band gaps in the relaxed linear micromorphic continuum.
	\newblock {\em Zeitschrift f\"ur Angewandte Mathematik und Mechanik},
	95(9):880--887, 2015.
	
	\bibitem{MadeoNeffGhibaW}
	A.~Madeo, P.~Neff, I.~D. Ghiba, L.~Placidi, and G.~Rosi.
	\newblock Wave propagation in relaxed linear micromorphic continua: modelling
	metamaterials with frequency band-gaps.
	\newblock {\em Continuum Mechanics and Thermodynamics}, 27:551--570, 2015.
	
	\bibitem{madeo2016reflection}
	A.~Madeo, P.~Neff, I.D. Ghiba, and G.~Rosi.
	\newblock Reflection and transmission of elastic waves in non-local band-gap
	metamaterials: a comprehensive study via the relaxed micromorphic model.
	\newblock {\em Journal of the Mechanics and Physics of Solids}, 95:441--479,
	2016.
	
	\bibitem{mielke2004uniqueness}
	A.~Mielke and Y.B. Fu.
	\newblock Uniqueness of the surface-wave speed: a proof that is independent of
	the {S}troh formalism.
	\newblock {\em Mathematics and Mechanics of Solids}, 9(1):5--15, 2004.
	
	\bibitem{mielke1998quasiconvexity}
	A.~Mielke and P.~Sprenger.
	\newblock Quasiconvexity at the boundary and a simple variational formulation
	of {A}gmon's condition.
	\newblock {\em Journal of Elasticity}, 51(1):23--41, 1998.
	
	\bibitem{Mindlin64}
	R.D. Mindlin.
	\newblock Micro-structure in linear elasticity.
	\newblock {\em Archive of Rational Mechanics and Analysis}, 16:51--77, 1964.
	
	\bibitem{NeffGhibaMadeoLazar}
	P.~Neff, I.~D. Ghiba, M.~Lazar, and A.~Madeo.
	\newblock The relaxed linear micromorphic continuum: well-posedness of the
	static problem and relations to the gauge theory of dislocations.
	\newblock {\em The Quarterly Journal of Mechanics and Applied Mathematics},
	68:53--84, 2015.
	
	\bibitem{NeffGhibaMicroModel}
	P.~Neff, I.~D. Ghiba, A.~Madeo, L.~Placidi, and G.~Rosi.
	\newblock A unifying perspective: the relaxed linear micromorphic continuum.
	\newblock {\em Continuum Mechanics and Thermodynamics}, 26:639--681, 2014.
	
	\bibitem{NunziatoCowin79}
	J.W. Nunziato and S.C. Cowin.
	\newblock A nonlinear theory of elastic materials with voids.
	\newblock {\em Archive of Rational Mechanics and Analysis}, 72:175--201, 1979.
	
	\bibitem{kundu2017love}
	D.K. Pandit, S.~Gupta, S.~Kundu, A.~Kumari.
	\newblock Love wave propagation in heterogeneous micropolar media.
	\newblock {\em Mechanics Research Communications}, 83:6--11, 2017.
	
	
	\bibitem{Rayleigh}	{Lord} Rayleigh, \newblock On waves propagated along the plane surface of an elastic solid. \newblock {\em Proceedings of the London Mathematical Society}, 1.1: 4-11, 1885.
	
	\bibitem{straughan2008stability}
	B.~Straughan.
	\newblock {\em Stability and {W}ave {M}otion in {P}orous {M}edia}, volume 165.
	\newblock Springer Science \& Business Media, 2008.
	
	\bibitem{stroh1962steady}
	A.N. Stroh.
	\newblock Steady state problems in anisotropic elasticity.
	\newblock {\em Journal of Mathematics and Physics}, 41(1-4):77--103, 1962.
	
\end{thebibliography}

\appendix

\section{Notation}\label{NS}
In the following, we recall some useful notations for the present work. For $ a,b\in \mathbb{R}^{3 \times 3} $ we let $\langle{a,b}\rangle_{\mathbb{R}^3}$ denote the scalar product on $ \mathbb{R}^3$ with associated vector norm $\norm{a}^2=\langle {a,a}\rangle$. We denote by $\mathbb{R}^{3\times 3}$   the set of real $3 \times 3$ second order tensors, written with capital letters. Matrices will be denoted by bold symbols, e.g. $\mathbf{X}\in \mathbb{R}^{3\times 3}$, while $X_{ij}$ will denote its component. The standard Euclidean product on $\mathbb{R}^{3 \times 3}$ is given by $\langle{ \mathbf{X},\mathbf{Y}}\rangle_{\mathbb{R}^{3 \times 3}}=\tr( \mathbf{X}\,\mathbf{Y}^T)$, and thus, the Frobenious tensor norm is $\norm{ \mathbf{X}}^2=\langle{ \mathbf{X}, \mathbf{X}}\rangle_{\mathbb{R}^{3 \times 3}}$. In the following we omit the index $\mathbb{R}^3, \mathbb{R}^{3\times 3}$. The identity tensor on $\mathbb{R}^{3\times 3}$ will be denoted by $\id$, so that $\tr( \mathbf{X})=\langle{ \mathbf{X},\id}\rangle$. We let $\Sym$  denote the set of symmetric tensors. We adopt the usual abbreviations of  Lie-algebra theory, i.e., $\mathfrak{so}(3):=\{ \mathbf{A}\,\in \mathbb{R}^{3\times 3}| \mathbf{A}^T=- \mathbf{A}\}$ is the Lie-algebra of skew-symmetric tensors and $\mathfrak{sl}(3):=\{ \mathbf{X}\, \in \mathbb{R}^{3\times 3} |\tr( \mathbf{X})=0 \}$  is the Lie-algebra of traceless tensors. For all $ \mathbf{X} \in\mathbb{R}^{3\times 3}$  we set $\sym\,  \mathbf{X}=\frac{1}{2}( \mathbf{X}^T+ \mathbf{X}) \in\Sym,\,\skw  \mathbf{X}=\frac{1}{2}( \mathbf{X}- \mathbf{X}^T)\in \mathfrak{so}(3)$ and the deviatoric (trace-free) part $\dev\, \mathbf{X}= \mathbf{X}-\frac{1}{3}\tr( \mathbf{X})\in\, \mathfrak{sl}(3)$ and we have the orthogonal Cartan-decomposition of the Lie-algebra $
\mathfrak{gl}(3)=\{\mathfrak{sl}(3)\cap \Sym(3)\}\oplus\mathfrak{so}(3)\oplus\mathbb{R}\cdot\id,\ 
\mathbf{X}= \dev\,\sym\, \mathbf{X}+\skw\, \mathbf{X}+\frac{1}{3}\tr( \mathbf{X})\,\id.\
$
We use the canonical identification of $\mathbb{R}^3$ with $\so(3)$, and, for
\begin{align}
	\mathbf{A}=	\begin{footnotesize}\begin{pmatrix}
			0 &-a_3&a_2\\
			a_3&0& -a_1\\
			-a_2& a_1&0
	\end{pmatrix}\end{footnotesize}\in \so(3)
\end{align}
we consider the operators $\axl\,:\so(3)\rightarrow\mathbb{R}^3$ and $\anti:\mathbb{R}^3\rightarrow \so(3)$ through
\begin{align}
	\axl\,( \mathbf{A}):&=\left(
	a_1,
	a_2,
	a_3
	\right)^T,\quad \quad  \mathbf{A}.\, v=(\axl\,  \mathbf{A})\times v, \qquad \qquad (\anti(v))_{ij}=-\epsilon_{ijk}\,v_k, \quad \quad \forall \, v\in\mathbb{R}^3,
	\notag \\(\axl\,  \mathbf{A})_k&=-\frac{1}{2}\, \epsilon_{ijk} \mathbf{A}_{ij}=\frac{1}{2}\,\epsilon_{kij} {A}_{ji}\,, \quad  {A}_{ij}=-\epsilon_{ijk}\,(\axl\,  \mathbf{A})_k=:\anti(\axl\,  \mathbf{A})_{ij},
\end{align}
where $\epsilon_{ijk}$ is the totally antisymmetric third order permutation tensor.

For a  regular enough function $f(t,x_1,x_2,x_3)$,  $f_{,t}$ denotes the derivative with respect to the time $t$, while  $ \frac{\partial\, f}{\partial \,x_i}$ denotes the $i$-component of the gradient $\nabla f$.  For vector fields $u=\left(    u_1, u_2, u_3\right)^T$ with  $u_i\in 
{\rm H}^1(\Omega)\,=\,\{u_i\in {\rm L}^2(\Omega)\, |\, \nabla\, u_i\in {\rm L}^2(\Omega)\}, $  $i=1,2,3$,
we define
$
\mathrm{D} u:=\left(
\nabla\,  u_1\,|\,
\nabla\, u_2\,|\,
\nabla\, u_3
\right)^T.
$
The corresponding Sobolev-space will be also denoted by
$
{\rm H}^1(\Omega)$.  In addition, for a tensor field
$\mathbf{P}$ with rows in ${\rm H}({\rm div}\,; \Omega)$, i.e.,
$
\mathbf{P}=\begin{footnotesize}\begin{footnotesize}\begin{pmatrix}
			\mathbf{P}^T.e_1\,|\,
			\mathbf{P}^T.e_2\,|\,
			\mathbf{P}^T\, e_3
\end{pmatrix}\end{footnotesize}\end{footnotesize}^T$ with $( \mathbf{P}^T.e_i)^T\in {\rm H}({\rm div}\,; \Omega):=\,\{v\in {\rm L}^2(\Omega)\, |\, {\rm div}\, v\in {\rm L}^2(\Omega)\}$, $i=1,2,3$,
we define
$
{\rm Div}\, \mathbf{P}:=\begin{footnotesize}\begin{footnotesize}\begin{pmatrix}
			{\rm div}\, ( \mathbf{P}^T.e_1)^T\,|\,
			{\rm div}\, ( \mathbf{P}^T.e_2)^T\,|\,
			{\rm div}\, ( \mathbf{P}^T\, e_3)^T
\end{pmatrix}\end{footnotesize}\end{footnotesize}^T
$
while for tensor fields $ \mathbf{P}$ with rows in ${\rm H}({\rm curl}\,; \Omega)$, i.e.,
$
\mathbf{P}=\begin{footnotesize}\begin{footnotesize}\begin{pmatrix}
			\mathbf{P}^T.e_1\,|\,
			\mathbf{P}^T.e_2\,|\,
			\mathbf{P}^T\, e_3
\end{pmatrix}\end{footnotesize}\end{footnotesize}^T$ with $( \mathbf{P}^T.e_i)^T\in {\rm H}({\rm curl}\,; \Omega):=\,\{v\in {\rm L}^2(\Omega)\, |\, {\rm curl}\, v\in {\rm L}^2(\Omega)\}
$, $i=1,2,3$,
we define
$
{\rm Curl}\, \mathbf{P}:=\begin{footnotesize}\begin{footnotesize}\begin{pmatrix}
			{\rm curl}\, ( \mathbf{P}^T.e_1)^T\,|\,
			{\rm curl}\, ( \mathbf{P}^T.e_2)^T\,|\,
			{\rm curl}\, ( \mathbf{P}^T\, e_3)^T
\end{pmatrix}\end{footnotesize}\end{footnotesize}^T
.
$

\end{footnotesize}

\end{document}